\documentclass[12pt]{article}
\usepackage{fullpage,amsmath,amsthm,amsfonts,amssymb,graphicx,setspace,amscd,float,hyperref,enumitem}

\newtheorem{thm}{Theorem}[section]
\newtheorem{lem}[thm]{Lemma}

\newtheorem{prop}[thm]{Proposition}

\theoremstyle{definition}
\newtheorem{df}[thm]{Definition}
\newtheorem{rk}[thm]{Remark}
\newtheorem{ex}[thm]{Example}

\newtheorem*{mainthm}{Theorem}

\newtheorem*{mainprop}{Proposition}

\begin{document}

\title{Ideal Whitehead Graphs in $Out(F_r)$ II:\\
The Complete Graph in Each Rank}
\author{Catherine Pfaff}
\date{}
\maketitle

\abstract{We show how to construct, for each $r \geq 3$, an ageometric, fully irreducible $\phi\in Out(F_r)$ whose ideal Whitehead graph is the complete graph on $2r-1$ vertices.

This paper is the second in a series of three where we show that precisely eighteen of the twenty-one connected, simplicial, five-vertex graphs are ideal Whitehead graphs of fully irreducible $\phi \in Out(F_3)$. The result is a first step to an $Out(F_r)$ version of the Masur-Smillie theorem proving precisely which index lists arise from singular measured foliations for pseudo-Anosov mapping classes.

In this paper we additionally give a method for finding periodic Nielsen paths and prove a criterion for identifying representatives of ageometric, fully irreducible $\phi\in Out(F_r)$}

\section{Introduction}

\noindent In \cite{ms93} Masur and Smillie list precisely which singularity index lists arise from the pair of invariant foliations for a pseudo-Anosov mapping class. The index lists were significant in their stratification of the space of quadratic differentials into strata invariant under the Teichmuller flow. Several papers studying the stratification include \cite{kz03}, \cite{l04}, \cite{l05}, and \cite{z10}. While $Out(F_r)$ index theory has been developed in papers such as \cite{gjll}, \cite{gl95}, \cite{ch10}, and \cite{ch12}, this is the first on index realization.

Our search for ideal Whitehead graphs arising from fully irreducible free group outer automorphisms is motivated by our goal of determining the $Out(F_r)$-version of the Masur-Smillie theorem. An ideal Whitehead graph (see Section \ref{Ch:PrelimDfns}) is a finer invariant than a singularity index list and encodes information about the attracting lamination for a fully irreducible outer automorphism.

We construct, for each $r \geq 3$, an ageometric, fully irreducible $\phi\in Out(F_r)$ whose ideal Whitehead graph is the complete ($2r-1$)-vertex graph. We consequently prove:

\begin{mainthm} \emph{Let $\mathcal{C}_r$ denote the complete ($2r-1$)-vertex graph. For each $r \geq 3$, there exists an ageometric, fully irreducible $\phi \in Out(F_r)$ such that $\mathcal{C}_r$ is the ideal Whitehead graph $\mathcal{IW}(\phi)$ for $\phi$.}
\end{mainthm}

That the ($2r-1$)-vertex complete graph occurs as an ideal Whitehead graph in each rank is both nonobvious and significant. First, the complete graph has no cut vertices. This phenomena holds significance for other versions of Whitehead graphs (see, for example, \cite{c10}, \cite{mm10}, \cite{s99}). Additionally, by \cite{hm11}, cut vertices in an ideal Whitehead graph have implications about periodic Nielsen paths. Second, the ($2r-1$)-vertex complete graph is a connected graph yielding the index sum $\frac{3}{2}-r$. This sum is as close as possible to that of $1-r$, achieved by \emph{geometrics} (fully irreducibles induced by pseudo Anosov surface homeomorphisms), without being achieved by a geometric outer automorphism. As in \cite{p12b}, we denote the set of connected ($2r-1$)-vertex simplicial graphs by $\mathcal{PI}_{(r; (\frac{3}{2}-r))}$. Third, the existence of such complicated ideal Whitehead graphs highlights significant depth within $Out(F_r)$ theory, beyond that of mapping class groups, as the ideal Whitehead graph of any pseudo-Anosov is simply a disjoint union of circles.

To show that our examples indeed represent fully irreducibles, we prove a folk lemma, Proposition \ref{P:FIC}, the ``Full Irreducibility Criterion (FIC).'' \cite{k12} gives another criterion inspired by our FIC.

\begin{mainprop}{\ref{P:FIC}} \textbf{(The Full Irreducibility Criterion)}
\emph{Let $g$ be a train track representing an outer automorphism $\phi \in Out(F_r)$ such that
~\\
\vspace{-6mm}
\begin{itemize}
\item [(I)] $g$ has no periodic Nielsen paths, \\[-6mm]
\item [(II)] the transition matrix for $g$ is Perron-Frobenius, and \\[-6mm]
\item [(III)] all local Whitehead graphs $\mathcal{LW}(x;g)$ for $g$ are connected.
\end{itemize}
\indent Then $\phi$ is fully irreducible.}
\end{mainprop}

We address three issues when applying the criterion to a train track representative $g\colon \Gamma \to \Gamma$. First, we must verify that $g$ has no ``periodic Nielsen paths.'' Recall \cite{bh92} that, for a train track map $g\colon \Gamma \to \Gamma$, a nontrivial path $\rho$ in $\Gamma$ is called a \emph{periodic Nielsen Path (pNp)} if, for some $k$, $g^k(\rho) \simeq \rho$ rel endpoints. We ensure our particular representatives are pNp-free using a ``Nielsen path prevention sequence,'' see Lemma \ref{L:NPKilling}. Proposition \ref{P:NPIdentification} provides a method for identifying pNp's of train track maps \emph{ideally decomposed} in the sense of \cite{p12b} (a method of a different nature can be found in \cite{t94}). This procedure can also be used to prove that an ideally decomposed representative has no pNp's. Second, we need that the local Whitehead graph $\mathcal{LW}(x;g)$ at each vertex $x$ of $\Gamma$ is connected, a condition satisfied in our case by the ideal Whitehead graph being connected. Third, the FIC includes a condition on the transition matrix (as defined in \cite{bh92}) for $g$, satisfied when some $g^n$ maps each edge of $\Gamma$ over each other edge of $\Gamma$. We use a ``switch sequence'' to ensure this property is satisfied (Lemmas \ref{L:SwitchPathSmooth} and \ref{L:SwitchSequence}).

Finally, to ensure our representatives actually have the correct ideal Whitehead graphs, the representatives are constructed using paths in the lamination train track structures of \cite{p12b}. The paths correspond to Dehn twist automosphisms $x \mapsto xw$ and construct the lamination, as do the Dehn twists in \cite{p88}. In \cite{cp10} Clay and Pettet also use Dehn twists to construct fully irreducibles, but focus on subgroups generated by powers of two Dehn twists for two filling cyclic splittings. Further construction methods for fully irreducible outer automorphisms can be found in \cite{kl10} and \cite{h09}.

In \cite{p12b} we gave, for each $r \geq 3$, examples of connected, simplicial ($2r-1$)-vertex graphs that are not the ideal Whitehead graph of any fully irreducible outer automorphism $\phi \in Out(F_r)$. In \cite{p12d} we will finish our proof that precisely eighteen of the twenty-one connected, simplicial, five-vertex graphs are ideal Whitehead graphs of fully irreducible $\phi \in Out(F_3)$. The results of this paper are used for proving the theorem in \cite{p12d}, but also make progress in a second direction, as we prove existence of the complete graph in each rank instead of focusing exclusively on rank-three, as we will in \cite{p12d}.

\subsection*{Acknowledgements}

\indent The author would like to thank Lee Mosher, as always, for truly invaluable conversations, Arnaud Hilion for advice, and Martin Lustig for continued interest in her work. She would also like to thank Michael Handel for his recommendation on how to complete the ``Full Irreducibility Criterion'' proof. Finally, she would like to thank the referee for a very thoughtful and helpful referee report. She extends her gratitude to Bard College at Simon's Rock and the CRM for their hospitality.

\section{Preliminary definitions and notation}{\label{Ch:PrelimDfns}}

\emph{\textbf{We continue with the introduction's notation. Additionally, unless otherwise stated, we assume throughout this document that outer automorphism representatives are train track (tt) representatives in the sense of \cite{bh92}.}}

For a rank $r$ free group \emph{$F_r$}, $\mathcal{FI}_r$ will denote the set of the fully irreducible elements of $Out(F_r)$.

\vskip10pt

\noindent \textbf{Directions and turns}

\vskip1pt

In general we use definitions from \cite{bh92} and \cite{bfh00} for discussing train track maps. We remind the reader of additional definitions and notation given in \cite{p12b}. Here $\Gamma$ will be a rose and $g\colon\Gamma\to\Gamma$ will represent $\phi \in Out(F_r)$. $\mathcal{E}^+(\Gamma):= \{E_1, \dots, E_{n}\}= \{e_1, e_1, \dots, e_{2n-1}, e_{2n} \}$ will be the edge set of $\Gamma$ with a prescribed orientation. $\mathcal{E}(\Gamma)$:$=\{E_1, \overline{E_1}, \dots, E_n, \overline{E_n} \}$, where $\overline{E_i}$ denotes $E_i$ oppositely oriented. If an edge indexing $\{E_1, \dots, E_{n}\}$ (thus indexing $\{e_1, e_1, \dots, e_{2n-1}, e_{2n} \}$) is prescribed, $\Gamma$ is called \emph{edge-indexed}. $\mathcal{D}(v)$ or $\mathcal{D}(\Gamma)$ will denote the set of directions at the vertex $v$. For each $e \in \mathcal{E}(\Gamma)$, $D_0(e)$ will denote the initial direction of $e$. Also, $D_0 \gamma := D_0(e_1)$ for any path $\gamma=e_1 \dots e_k$ in $\Gamma$. \emph{$Dg$} will denote the direction map $g$ induces. We call $d \in \mathcal{D}(\Gamma)$ \emph{periodic} if $Dg^k(d)=d$ for some $k>0$ and \emph{fixed} if $k=1$.

$\mathcal{T}(v)$ will denote the set of turns at $v$ and $D^tg$ the induced turn map. Sometimes we abusively write $\{\overline{e_i}, e_j\}$ for $\{D_0(\overline{e_i}), D_0(e_j)\}$. For a path $\gamma=e_1e_2 \dots e_{k-1}e_k$ in $\Gamma$, we say $\gamma$ \emph{traverses} $\{\overline{e_i}, e_{i+1}\}$ for each $1 \leq i < k$. Recall that a turn is called \emph{illegal} for $g$ if $Dg^k(d_i)=Dg^k(d_j)$ for some $k$.

\vskip10pt

\noindent \textbf{Transition matrices and irreducibility.}

\vskip1pt

The \emph{transition matrix} for a topological representative $g$ is the square matrix where the $ij^{th}$ entry is the number of times $g(E_j)$ traverses $E_i$ in either direction. A matrix $A=[a_{ij}]$ is \emph{irreducible} if each entry $a_{ij} \geq 0$ and if, for each $i$ and $j$, there exists a $k>0$ so that the $ij^{th}$ entry of $A^k$ is strictly positive. The matrix is \emph{Perron-Frobenius (PF)} if each sufficiently high $k$ works for all index pairs $\{i, j \}$, in which case the map is called \emph{expanding}.

The Full Irreducibility Criterion will require that the transition matrix for a representative be PF. It will be relevant that any power of a PF matrix is both PF and irreducible. Additionally, a topological representative is irreducible if and only if its transition matrix is irreducible \cite{bh92}.

\vskip10pt

\noindent \textbf{Periodic Nielsen paths and ageometric outer automorphisms.}

\vskip1pt

Recall \cite{bh92}, for $g\colon\Gamma\to\Gamma$, a nontrivial path $\rho$ in $\Gamma$ is called a \emph{periodic Nielsen path (pNp)} if, for some $k$, $g^k(\rho)\simeq\rho$ rel endpoints. For $k=1$, $\rho$ is called a \emph{Nielsen path (Np)}. $\rho$ is called an \emph{indivisible Nielsen path (iNp)} if it cannot be written as a nontrivial concatenation $\rho=\rho_1 \cdot \rho_2$ of Np's $\rho_1$ and $\rho_2$. If $\rho$ is an iNp for an expanding irreducible train track map $g$, then (Lemma 3.4, \cite{bf94}) there exist unique, nontrivial, legal paths $\alpha$, $\beta$, and $\tau$ in $\Gamma$ so that $\rho=\bar{\alpha}\beta$, $g(\alpha)= \tau\alpha$, and $g(\beta)=\tau\beta$. In \cite{bf94}, immersed paths $\alpha_1, \dots, \alpha_k \in \Gamma$ are said to form an \emph{orbit of periodic Nielsen paths} if $g^k(\alpha_i) \simeq \alpha_{\text{i+1 mod k}}$ rel endpoints, for all $1 \leq i \leq k$. The orbit is called \emph{indivisible} when $\alpha_1$ is not a concatenation of subpaths belonging to orbits of pNps. Each $\alpha_i$ in an indivisible orbit is called an \emph{indivisible periodic Nielsen path} (\emph{ipNp}).

Recall \cite{gjll}, that an outer automorphism is \emph{ageometric} whose stable representative, in the sense of \cite{bh92}, has no pNp's. We denote by $\mathcal{AFI}_r$ the subset of $\mathcal{FI}_r$ consisting of the ageometric elements.

\vskip10pt

\noindent \textbf{Ideal Whitehead graphs and lamination train track (ltt) structures.}

\vskip1pt

We remind the reader of an \cite{hm11} ideal Whitehead graph definition and \cite{p12b} lamination train track structure definition. See \cite{p12a} and \cite{hm11} for descriptions of ideal Whitehead graph alternative definitions and outer automorphism invariance. Note that, while we use a representative to construct it, the ideal Whitehead graph does not depend on the choice of representative for an outer automorphism.

Let $\Gamma$ be a marked graph, $g\colon \Gamma \to \Gamma$ a tt representive of $\phi \in Out(F_r)$, and $v \in \Gamma$ a \emph{singularity} (the endpoint of an ipNp or a vertex with at least three periodic directions). The \emph{local Whitehead graph} $\mathcal{LW}(g; v)$ for $g$ at $v$ has:

(1) a vertex for each direction $d \in \mathcal{D}(v)$ and

(2) edges connecting vertices for $d_1, d_2 \in \mathcal{D}(v)$ when $\{d_1, d_2 \}$ is taken by some $g^k(e)$, with $e \in \mathcal{E}(\Gamma)$.

\noindent The \emph{local stable Whitehead graph} $\mathcal{SW}(g; v)$ is the subgraph obtained by restricting precisely to vertices with periodic direction labels. For a rose $\Gamma$ with vertex $v$, we denote the single local stable Whitehead graph $\mathcal{SW}(g; v)$ by $\mathcal{SW}(g)$ and the single local Whitehead graph $\mathcal{LW}(g; v)$ by $\mathcal{LW}(g)$.

For a pNp-free $g$, the \emph{ideal Whitehead graph $\mathcal{IW}(\phi)$ of $\phi$} is isomorphic to $\underset{\text{singularities v} \in \Gamma}{\bigsqcup} \mathcal{SW}(g;v)$. In particular, when $\Gamma$ is a rose, $\mathcal{IW}(\phi) \cong \mathcal{SW}(g)$.

Let $g$ be a pNp-free tt map on a marked rose $\Gamma$ with vertex $v$. Recall from \cite{p12b} the definition of the \emph{lamination train track (ltt) structure} $G(g)$ for $g$: The \emph{colored local Whitehead graph $\mathcal{CW}(g)$ at $v$} is $\mathcal{LW}(g)$ with the subgraph $\mathcal{SW}(g)$ colored purple and $\mathcal{LW}(g)- \mathcal{SW}(g)$ red (nonperiodic direction vertices are red). Let $\Gamma_N=\Gamma-N(v)$, where $N(v)$ is a contractible neighborhood of $v$. For each $E_i \in \mathcal{E}^+$, add vertices $D_0(E_i)$ and $\overline{D_0(E_i)}$ at the corresponding boundary points of the partial edge $E_i-(N(v) \cap E_i)$. $G(g)$ is formed from $\Gamma_N \bigsqcup \mathcal{CW}(g)$ by identifying the vertex $d_i$ in $\Gamma_N$ with the vertex $d_i$ in $\mathcal{CW}(g)$. Nonperiodic directions vertices are red, edges of $\Gamma_N$ are black, and periodic vertices are purple.

$G(g)$ is given a \emph{smooth structure} via a partition of the edges into the set of black edges $\mathcal{E}_b$ and the set of colored edges $\mathcal{E}_c$. A \emph{smooth path} will mean a path alternating between colored and black edges.

We refer the reader to \cite{p12a} or \cite{p12b} for a thorough presentation of abstract lamination train track structures. We summarize just several definitions here.

Recall that a \emph{train track (tt) graph} is a finite graph $G$ satisfying:
~\\
\vspace{-6mm}
\begin{description}
\item[tt1:] $G$ has no valence-1 vertices; \\[-6mm]
\item[tt2:] each edge of $G$ has 2 distinct vertices (single edges are never loops); and \\[-6mm]
\item[tt3:] the edge set of $G$ is partitioned into two subsets, $\mathcal{E}_b$ (the ``black'' edges) and $\mathcal{E}_c$ (the ``colored'' edges), such that each vertex is incident to at least one $E_b \in \mathcal{E}_b$ and at least one $E_c \in \mathcal{E}_c$. \\[-6mm]
\end{description}

A \emph{lamination train track (ltt) structure $G$} is a pair-labeled colored train track graph (black edges will be included, but not considered colored) satisfying:
~\\
\vspace{-6mm}
\begin{description}
\item[ltt1:] Vertices are either purple or red. \\[-6mm]
\item[ltt2:] Edges are of 3 types ($\mathcal{E}_b$ comprises the black edges and $\mathcal{E}_c$ comprises the red and purple edges): \\[-6mm]
~\\
\vspace{-\baselineskip}
\indent \indent {\begin{description}
\item[(Black Edges):] A single black edge connects each pair of (edge-pair)-labeled vertices.  There are no other black edges. In particular, each vertex is contained in a unique black edge. \\[-5.7mm]
\item[(Red Edges):] A colored edge is red if and only if at least one of its endpoint vertices is red. \\[-5.7mm]
\item[(Purple Edges):] A colored edge is purple if and only if both endpoint vertices are purple. \\[-6mm]
    \end{description}}
\item[ltt3:] No pair of vertices is connected by two distinct colored edges. \\[-6mm]
    \end{description}

We denote the purple subgraph of $G$ (from $\mathcal{SW}(g)$) by \emph{$\mathcal{PI}(G)$} and, if $\mathcal{G} \cong \mathcal{PI}(G)$, say $G$ \emph{is an ltt structure for $\mathcal{G}$}. An \emph{$(r;(\frac{3}{2}-r))$ ltt structure} is an ltt structure $G$ for a $\mathcal{G} \in \mathcal{PI}_{(r;(\frac{3}{2}-r))}$ such that:
~\\
\vspace{-\baselineskip}
\begin{description}
\item[ltt(*)4:] $G$ has precisely 2r-1 purple vertices, a unique red vertex, and a unique red edge. \\[-6mm]
\end{description}

We consider ltt structures \emph{equivalent} that differ by an ornamentation-preserving homeomorphism and refer the reader to the Standard Notation and Terminology 2.2 of \cite{p12b}. In particular, in abstract and nonabstract ltt structures, [$d_i, d_j$] is the edge connecting a vertex pair $\{d_i, d_j \}$, $[e_i]$ denotes the black edge [$d_i, \overline{d_i}$] for $e_i \in \mathcal{E}(\Gamma)$, and \emph{$\mathcal{C}(G)$} denotes the colored subgraph (from $\mathcal{LW}(g)$). Purple vertices are \emph{periodic} and red vertices \emph{nonperiodic}. $G$ is \emph{admissible} if birecurrent as a train track structure (i.e has a locally smoothly embedded line traversing each edge infinitely many times as $\mathbb{R}\to \infty$ and as $\mathbb{R}\to -\infty$).

For an $(r;(\frac{3}{2}-r))$ ltt structure $G$ for $\mathcal{G}$, additionally: \newline
\noindent \textbf{1.} $d^u$ labels the unique red vertex and is called the \emph{unachieved direction}. \newline
\noindent \textbf{2.} $e^R=[t^R]$ denotes the unique red edge, $\overline{d^a}$ labels its purple vertex, thus $t^R= \{d^u, \overline{d^a} \}$ ($e^R= [d^u, \overline{d^a}]$). \newline
\noindent \textbf{3.} $\overline{d^a}$ is contained in a unique black edge, which we call the \emph{twice-achieved edge}. \newline
\noindent \textbf{4.} $d^a$ will label the other twice-achieved edge vertex and be called the \emph{twice-achieved direction}. \newline
\noindent \textbf{5.} If $G$ has a subscript, the subscript carries over to all relevant notation. For example, in $G_k$, $d^u_k$ will label the red vertex and $e^R_k$ the red edge.

We call a $2r$-element set of the form $\{x_1, \overline{x_1}, \dots, x_r, \overline{x_r} \}$, with elements paired into \emph{edge pairs} $\{x_i, \overline{x_i}\}$, a \emph{rank}-$r$ \emph{edge pair labeling set} (we write $\overline{\overline{x_i}}=x_i$). We call a graph with vertices labeled by an edge pair labeling set a \emph{pair-labeled} graph, and an \emph{indexed pair-labeled} graph if an indexing is prescribed. A $\mathcal{G} \in \mathcal{PI}_{(r;(\frac{3}{2}-r))}$ is \emph{(index) pair-labeled} whose vertices are labeled by a $2r-1$ element subset of the rank-$r$ (indexed) edge pair labeling set. An ltt structure, index pair-labeled as a graph, is an \emph{indexed pair-labeled} ltt structure if the vertices of the black edges are indexed by edge pairs. Index pair-labeled ltt structures are \emph{equivalent} that are equivalent as ltt structures via an equivalence preserving the indexing of the vertex labeling set. By rank-$r$ index pair-labeling an $(r;(\frac{3}{2}-r))$ ltt structure $G$ and edge-indexing the edges of an $r$-petaled rose $\Gamma$, one creates an identification of the vertices in $G$ with $\mathcal{D}(v)$, where $v$ is the vertex of $\Gamma$. With this identification, we say $G$ is \emph{based} at $\Gamma$. In such a case we may use the notation $\{d_1, d_2, \dots, d_{2r-1}, d_{2r} \}$ for the vertex labels. Additionally, $[e_i]$ denotes $[D_0(e_i), D_0(\overline{e_i})] = [d_i, \overline{d_i}]$ for each edge $e_i \in \mathcal{E}(\Gamma)$. We call a permutation of the indices $1 \leq i \leq 2r$ combined with a permutation of the elements of each pair $\{x_i, \overline{x_i}\}$ an \emph{edge pair (EP) permutation}. Edge-indexed graphs will be considered \emph{edge pair permutation (EPP) isomorphic} if there is an EP permutation making the labelings identical.

\vskip10pt

\noindent \textbf{Ideal decompositions.}

\vskip1pt

M. Feighn and M. Handel defined rotationless train track representatives and outer automorphisms in \cite{fh11}. Recall \cite{hm11}: Let a $\phi \in \mathcal{AFI}_r$ be such that $\mathcal{IW}(\phi) \in \mathcal{PI}_{(r; (\frac{3}{2}-r))}$, then $\phi$ is \emph{rotationless} if and only if the vertices of $\mathcal{IW}(\phi) \in \mathcal{PI}_{(r; (\frac{3}{2}-r))}$ are fixed by the action of $\phi$.

The following is Proposition 3.3 of \cite{p12b}:

\begin{mainprop}{\label{P:IdealDecomposition}} \cite{p12b} Let $\phi \in \mathcal{AFI}_r$ with $\mathcal{IW}(\phi) \in \mathcal{PI}_{(r; (\frac{3}{2}-r))}$. There exists a pNp-free tt map on the rose  representing a rotationless power $\psi=\phi^R$ and decomposing as $\Gamma_0 \xrightarrow{g_1} \Gamma_1 \xrightarrow{g_2} \cdots \xrightarrow{g_{n-1}}
\Gamma_{n-1} \xrightarrow{g_n} \Gamma_n$, such that: \newline
\noindent (I) the index set $\{1, \dots, n \}$ is viewed as the set $\mathbf {Z}$/$n \mathbf {Z}$ with its natural cyclic ordering; \newline
\noindent (II) each $\Gamma_k$ is an edge-indexed rose with an indexing $\{e_{(k,1)}, e_{(k,2)}, \dots, e_{(k,2r-1)}, e_{(k,2r)}\}$ where:
~\\
\vspace{-\baselineskip}
{\begin{description}
\item (a) one can edge-index $\Gamma$ with $\mathcal{E}(\Gamma)=\{e_1, e_2, \dots, e_{2r-1}, e_{2r} \}$ such that, for each $t$ with $1 \leq t \leq 2r$, $g(e_t)=e_{i_1} \dots e_{i_s}$ where $(g_n \circ \cdots \circ g_1)(e_{0,t})=e_{n, i_1} \dots e_{n, i_s}$; \\[-5mm]
\item (b) for some $i_k,j_k$ with $e_{k,i_k} \neq (e_{k,j_k})^{\pm 1}$,
$$g_k(e_{k-1,t}):=
\begin{cases}
e_{k,i_k} e_{k,t} \text{ for $t=i_k$} \\
e_{k,t} \text{ for all $e_{k-1,t} \neq e_{k-1,j_k}^{\pm 1}$; and}
\end{cases}$$
\item (c) for each $e_t \in \mathcal{E}(\Gamma)$ such that $t \neq j_n$, we have $Dh(d_t)=d_t$, where $d_t=D_0(e_t)$.
\end{description}}
\end{mainprop}

Recall that tt maps satisfying (I)-(II) are called \emph{ideally decomposable ($\mathcal{ID}$)} with an \emph{ideal decomposition ($\mathcal{ID}$)}. An $\mathcal{ID}$ $g$ such that $\phi \in \mathcal{AFI}_r$ and $\mathcal{IW}(\phi) \in \mathcal{PI}_{(r;(\frac{3}{2}-r))}$ has \emph{type $(r;(\frac{3}{2}-r))$}. In \cite{p12b} we proved for a $(r;(\frac{3}{2}-r))$ tt map $g:\Gamma \to \Gamma$, that $G(g)$ is an $(r;(\frac{3}{2}-r))$ ltt structure with base $\Gamma$.

Again we denote $e_{k-1,j_k}$ by $e^{pu}_{k-1}$, denote $e_{k,j_k}$ by $e^u_k$, denote $e_{k,i_k}$ by $e^a_k$, and denote $e_{k-1,i_{k-1}}$ by $e^{pa}_{k-1}$. Also,$\mathcal{D}_k := \mathcal{D}(\Gamma_k)$, $\mathcal{E}_k:=\mathcal{E}(\Gamma_k)$, and $G_k:=G(f_k)$ where $f_k:= g_k \circ \cdots \circ g_1 \circ g_n \circ  \cdots \circ g_{k+1}: \Gamma_k \to \Gamma_k$. And
$$g_{k,i}:=
\begin{cases}
g_k \circ \cdots \circ g_i\colon \Gamma_{i-1} \to \Gamma_k \text{ if $k>i$}\text{ and } \\
 g_k \circ \cdots \circ g_1 \circ g_n \circ  \cdots \circ g_i \text{ if $k<i$}
\end{cases}.$$
It is proved in \cite{p12b} that $D_0(e^u_k)=d^u_k$, $D_0(e^a_k)=d^a_k$, $D_0(e^{pu}_{k-1})=d^{pu}_{k-1}$, and $D_0(e^{pa}_{k-1})=d^{pa}_{k-1}$. As described in \cite{p12b}, for any $k,l$, we have a \emph{direction map} $Dg_{k,l}$, an \emph{induced map of turns} $Dg_{k,l}^t$, and an \emph{induced map of ltt structures} $Dg_{k,l}^T: G_{l-1} \mapsto G_k$. $Dg_{k,l}^C$ denotes the restriction to $\mathcal{C}(G_{l-1})$ of $Dg_{k,l}^T$.

\vskip10pt

\noindent \textbf{Extensions and switches.}

\vskip1pt

By a proper full fold we mean the identification of a (proper) partial edge with a full edge. A \emph{triple} $(g_k, G_{k-1}, G_k)$ is an ordered set of three objects where $g_k: \Gamma_{k-1} \to \Gamma_k$ is a proper full fold of roses and, for $i=k-1,k$, $G_i$ is an ltt structure with base $\Gamma_i$. Recall \cite{p12b}, in an $\mathcal{ID}$ of a $(r;(\frac{3}{2}-r))$ representative, each $(g_k, G_{k-1}, G_k)$ satisfies the ``admissible map properties'' $\mathcal{AM}$I-VII of \cite{p12b} and is either a ``switch'' or ``extension.''

A \emph{generating triple (gt)} is a triple $(g_k, G_{k-1}, G_k)$ where
~\\
\vspace{-\baselineskip}
\begin{description}
\item [(gtI)] $g_k: \Gamma_{k-1} \to \Gamma_k$ is a proper full fold of edge-indexed roses defined by
\begin{itemize}
\item [a.] $g_k(e_{k-1,j_k})= e_{k,i_k} e_{k,j_k}$  where $d^a_k=D_0(e_{k,i_k})$, $d^u_k=D_0(e_{k,j_k})$, and $e_{k,i_k} \neq (e_{k,j_k})^{\pm 1}$ and \\[-5mm]
\item [b.] $g_k(e_{k-1,t})= e_{k,t}$ for all $e_{k-1,t} \neq (e_{k,j_k})^{\pm 1}$;
\end{itemize}
\item [(gtII)] $G_i$ is an indexed pair-labeled $(r;(\frac{3}{2}-r))$ ltt structure with base $\Gamma_i$ for $i=k-1,k$; and
\item [(gtIII)] The induced map of based ltt structures $D^T(g_k): G_{k-1} \to G_k$ exists and, in particular, restricts to an isomorphism from $\mathcal{PI}(G_{k-1})$ to $\mathcal{PI}(G_k)$.
\end{description}

\vskip10pt

\noindent $G_{k-1}$ is the \emph{source ltt structure} and $G_k$ the \emph{destination ltt structure}. If both are admissible, the triple is \emph{admissible}. We sometimes write $g_k: e^{pu}_{k-1} \mapsto e^a_k e^u_k$ for $g_k$, write $d^{pu}_{k-1}$ for $d_{k-1,j_k}$, and write $e^{pa}_{k-1}$ for $e_{k-1,i_k}$. If $G_k$ and $G_{k-1}$ are index pair-labeled $(r;(\frac{3}{2}-r))$ ltt structures for $\mathcal{G}$, then $(g_k, G_{k-1}, G_k)$ will be a \emph{generating triple for $\mathcal{G}$}.

The \emph{switch} determined by a purple edge $[d^a_k, d_{(k,l)}]$ in $G_k$ is the gt $(g_k, G_{k-1}, G_k)$ for $\mathcal{G}$ satisfying:
~\\
\vspace{-\baselineskip}
\indent \begin{description}
\item[(swI):] $D^T(g_k)$ restricts to an isomorphism from $\mathcal{PI}(G_{k-1})$ to $\mathcal{PI}(G_k)$ defined by $d^{pu}_{k-1} \mapsto d^a_k=d_{k, i_k}$
    ($d_{k-1,t} \mapsto d_{k,t}$ for $d_{k-1,t} \neq d^{pu}_{k-1}$) and extended linearly over edges. \\[-5mm]
\item[(swII):] $d^{pa}_{k-1} = d^u_{k-1}$. \\[-5mm]
\item[(swIII):] $\overline{d^a_{k-1}} = d_{k-1,l}$.
\end{description}
~\\
\vspace{-13mm}
\begin{figure}[H]
\centering
\includegraphics[width=4.6in]{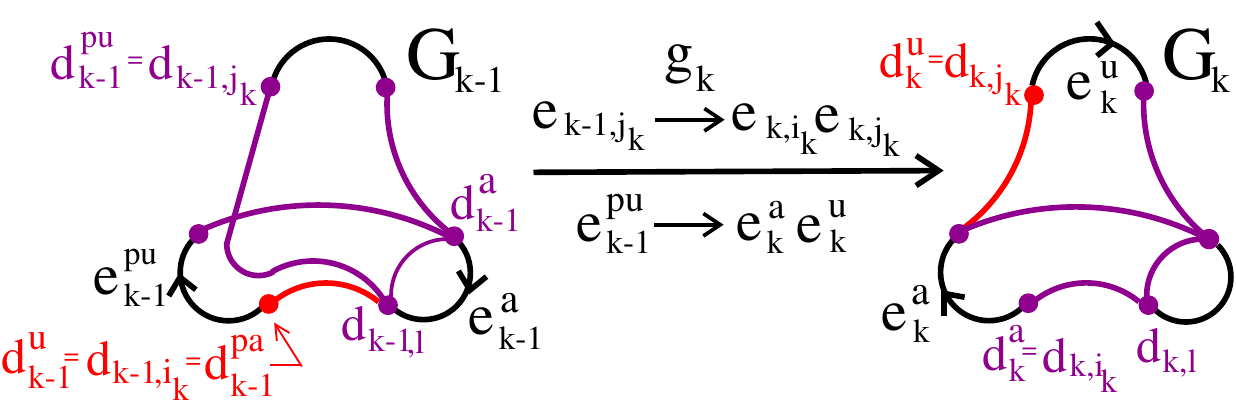}
\label{fig:SwitchDiagram} \\[-2mm]
\end{figure}

The \emph{extension determined by} $[d^a_k, d_{k,l}]$, is the gt $(g_k, G_{k-1}, G_k)$ for $\mathcal{G}$ satisfying:
~\\
\vspace{-\baselineskip}
\begin{description}
\item[(extI):] The restriction of $D^T(g_k)$ to $\mathcal{PI}(G_{k-1})$ is defined by sending, for each $j$, the vertex labeled $d_{k-1,j}$ to the vertex labeled $d_{k,j}$ and extending linearly over edges. \\[-5mm]
\item[(extII):] $d^u_{k-1}= d^{pu}_{k-1}$, i.e. $d^{pu}_{k-1}= d_{k-1,j_k}$ labels the single red vertex in $G_{k-1}$. \\[-5mm]
\item[(extIII):] $\overline{d^a_{k-1}}= d_{k-1,l}$.
\end{description}
~\\
\vspace{-13mm}
\begin{figure}[H]
\centering
\includegraphics[width=4.6in]{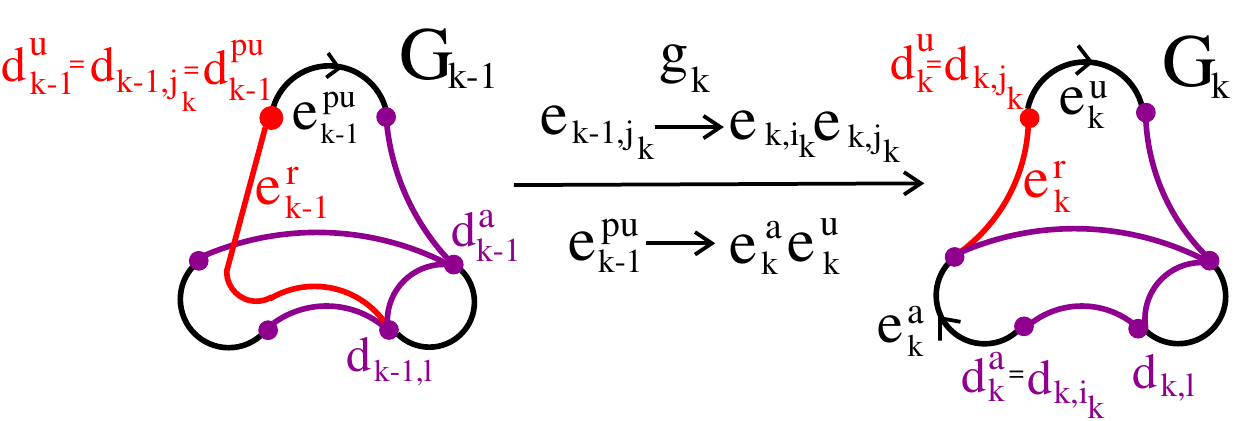}
\label{fig:ExtensionDiagram} \\[-3mm]
\end{figure}

\section{Compositions of extensions and switches}{\label{Ch:ConstructionCompositionsandSwitchPaths}}

Compositions of a sequence of extensions or a sequence of switches play an important role in our proofs.

\begin{df}{\label{D:AdmissibleCompositions}} A \emph{preadmissible composition} $(g_{i-k}, \dots, g_i, G_{i-k-1}, \dots, G_i)$ for a $\mathcal{G} \in \mathcal{PI}_{(r; (\frac{3}{2}-r))}$ is a sequence of proper full folds of (edge-pair)-indexed roses,
~\\
\vspace{-1.5mm}
$$\Gamma_{i-k-1} \xrightarrow{g_{i-k}} \Gamma_{i-k}
 \cdots \xrightarrow{g_{i-1}}\Gamma_{i-1} \xrightarrow{g_i} \Gamma_i,$$
with \emph{associated sequence of $(r;(\frac{3}{2}-r))$ ltt structures} for $\mathcal{G}$,
~\\
\vspace{-1mm}
$$G_{i-k-1} \xrightarrow{D^T(g_{i-k})} G_{i-k} \xrightarrow{D^T(g_{i-k+1})} \cdots \xrightarrow{D^T(g_{i-1})} G_{i-1} \xrightarrow{D^T(g_i)} G_i,$$
where, for each $i-k-1 \leq j < i$, $(g_{j+1}, G_j, G_{j+1})$ is an extension or switch for $\mathcal{G}$.
\end{df}

The Definition \ref{D:AdmissibleCompositions} notation will be standard. A composition is \emph{admissible} if each $G_j$ is. We call $g_{i,i-k}$ the \emph{associated automorphism}, $G_{i-k-1}$ the \emph{source ltt structure}, and $G_k$ the \emph{destination ltt structure}.

To ensure $\mathcal{IW}(g) \cong C_r$ in Theorem \ref{T:MainTheorem}, we use ``building block'' compositions of extensions: If each $(g_j, G_{j-1}, G_j)$ with $i-k < j \leq i$ is an admissible extension and $(g_{i-k}, G_{i-k-1}, G_{i-k})$ is an admissible switch, then we call $(g_{i-k}, \dots, g_i; G_{i-k-1}, \dots, G_i)$ an \emph{admissible construction composition} for $\mathcal{G}$. We call $g_{i,i-k}$ a \emph{construction automorphism}. Leaving out the switch, gives a \emph{purified construction automorphism} $g_p=g_i \circ \dots \circ g_{i-k+1}$ and \emph{purified construction composition} $(g_{i-k+1}, \dots, g_i; G_{i-k}, \dots, G_i)$.

A construction automorphism always has the form of a Dehn twist automorphism $e^{pu}_{i-k-1} \mapsto w e^u_{i-k}$, where $w=e^a_{i-k} \dots e^a_i$. One can view the composition as twisting the edge corresponding to $e^{pu}_{i-k-1}$ around the path corresponding to $w$ in the destination ltt structure. In the next section we describe these paths and prove (Proposition \ref{P:PathConstruction}) they ``construct'' a smooth path in their destination ltt structure.

\subsection{Construction Paths}{\label{S:ConstructionCompositions}}

\noindent Corresponding to a construction composition is a path in its destination ltt structure. A key property of such a path (Lemma \ref{P:PathConstruction}) holds when the construction composition is part of the ideal decomposition of a type $(r;(\frac{3}{2}-r))$ representative $g$: the image of the path's purple edges live in $G(g)$.

We abuse notation throughout this section by dropping indices. While not necessary, it may aid in visualization of the properties and procedures, as well as reduce potential confusion over indices.

\begin{lem}{\label{L:ConstructionPathSmooth}} Let $(g_1, \dots, g_n, G_0, \dots, G_n)$ be an $\mathcal{ID}$ for a $\mathcal{G} \in \mathcal{PI}_{(r; (\frac{3}{2}-r))}$ and $(g_{i-k}, \dots, g_i; G_{i-k-1}, \dots, G_i)$ a construction composition. Then
~\\
\vspace{-2mm}
$$[d^u_i, \overline{d^a_i}, d^a_i, d_i, d^a_{i-1}, d_{i-1}, \dots, d^a_{i-k+1}, d_{i-k+1}, d^a_{i-k}] =[d^u_i, \overline{d^a_i}, d^a_i, \overline{d^a_{i-1}}, \dots, \overline{d^a_{i-k}}, d^a_{i-k}]$$
is a smooth path in the ltt structure $G_i$.
\end{lem}

\begin{proof} We proceed by induction for decreasing $s$. Proof by induction is valid, as the proof does not rely on $G_{i-k-1}$ (the only thing distinguishing $(g_{i-k}, G_{i-k-1}, G_{i-k})$ as a switch). For the base case note that $e^R_i=[d^u_i, \overline{d^a_i}]$. So $[d^u_i, \overline{d^a_i}, d^a_i]$ is a path in $G_i$ and smooth, as it alternates between colored and black edges ($[d^u_i, \overline{d^a_i}]$ is colored and $[\overline{d^a_i}, d^a_i]$ is black). For the induction assume, for $i > s > i-k$, $[d^u_i, \overline{d^a_i}, d^a_i, \overline{d^a_{i-1}}, d^a_{i-1}, \dots, d^a_{s+1}, \overline{d^a_s}, d^a_s]$ is a smooth path in $G_i$ (ending with the black edge $[\overline{d^a_s}, d^a_s]$).

\indent By \cite{p12b} Corollary 5.6b, $e^R_{s-1}=[d^u_{s-1}, \overline{d^a_{s-1}}]$. By \cite{p12b} Lemma 5.7, $D^Cg_s([d^u_{s-1}, \overline{d^a_{s-1}}])$ \newline
\noindent $=[d^a_s, \overline{d^a_{s-1}}]$ is a purple edge in $G_s$. Since purple edges are always mapped to themselves by extensions (in the sense that $D^C$ preserves the second index of their vertex labels) and $D^Cg_{s}([d^u_{s}, \overline{d^a_{s-1}}])=[d^a_{s}, \overline{d^a_{s-1}}]$ is in $\mathcal{PI}(G_s)$, $D^Cg_{n,s}(\{d^u_{s-1}, \overline{d^a_{s-1}} \}) = D^Cg_{n,s+1}(D^Cg_{s}([d^u_{s-1}, \overline{d^a_{s-1}}]))= D^Cg_{n,s+1}([d^a_s, \overline{d^a_{s-1}}])=[d^a_s, \overline{d^a_{s-1}}]$ is in $\mathcal{PI}(G_i)$. Thus, including the purple edge $[d^a_s, \overline{d^a_{s-1}}]$ in the smooth path $[d^u_i, \overline{d^a_i}, d^a_i, \overline{d^a_{i-1}}, d^a_{i-1}, \dots, d^a_{s+1}, \overline{d^a_s}, d^a_s]$ gives the smooth path $[d^u_i, \overline{d^a_i}, d^a_i, \overline{d^a_{i-1}}, \dots, d^a_{s+1}, \overline{d^a_s}, d^a_s, \overline{d^a_{s-1}}]$. (It is smooth, as we added a colored edge to a path with edges alternating between colored and black, ending with black). By including the black edge $[\overline{d^a_{s-1}}, d^a_{s-1}]$ we get the construction path $[d^u_i, \overline{d^a_i}, d^a_i, \overline{d^a_{i-1}}, d^a_{i-1}, \dots, d^a_s, \overline{d^a_{s-1}}, d^a_{s-1}]$. (Also smooth, as we added a black edge to a path with edges alternating between colored and black, ending colored). This concludes the inductive step, hence proof. \qedhere
\end{proof}

The path of Lemma \ref{L:ConstructionPathSmooth} (depicted in Example \ref{E:ConstructionAutomorphism}) is called the \emph{construction path} for $(g_{i-k}, \dots,g_i;$ \newline
\noindent $G_{i-k-1}, \dots, G_i)$ and denoted $\gamma_{g_{i,i-k}}$. One obtains it by traversing the red edge $[d^u_i, \overline{d^a_i}]$ from the red vertex $d^u_i$ to the vertex $\bar d^a_i$, the black edge $[\overline{d^a_i}, d^a_i]$ from $\overline{d^a_i}$ to $d^a_i$, the extension determining purple edge $[d^a_i, d_i]=[d^a_i, \overline{d^a_{i-1}}]$ from $d^a_i$ to $d_i= \overline{d^a_{i-1}}$, the black edge $[\overline{d^a_{i-1}}, d^a_{i-1}]$ from $\overline{d^a_{i-1}}$ to $d^a_{i-1}$, the extension determining purple edge $[d^a_{i-1}, d_{i-1} ]= [d^a_{i-1}, \overline{d^a_{i-2}}]$ from $d^a_{i-1}$ to $d_{i-1}= \overline{d^a_{i-2}}$, the black edge $[\overline{d^a_{i-2}}, d^a_{i-2}]$ from $\overline{d^a_{i-2}}$ to $d^a_{i-2}$, continuing as such through the purple edges determining each $g_j$ (inserting black edges between), and finally traversing $[d^a_{i-k+1}, d_{i-k+1}]= [d^a_{i-k+1}, \overline{d^a_{i-k}}]$ and then $[\overline{d^a_{i-k}}, d^a_{i-k}]$ from $\overline{d^a_{i-k}}$ to $d^a_{i-k}$.

Let $G$ be an admissible $(r;(\frac{3}{2}-r))$ ltt structure with the standard notation. The \emph{construction subgraph} $G_C$ is constructed from $G$ via the following procedure:
\begin{enumerate}[itemsep=-1.5pt,parsep=3pt,topsep=3pt]
\item Remove the interior of the black edge $[e^u]$, the purple vertex $\overline{d^u}$, and the interior of any purple edges containing the vertex $\overline{d^u}$. Call the graph with these edges and vertices removed $G^1$.
\item Given $G^{j-1}$, recursively define $G^j$: Let $\{\alpha_{j-1,i} \}$ be the set of vertices in $G^{j-1}$ not contained in any colored edge of $G^{j-1}$. $G^j$ is obtained from $G^{j-1}$ by removing all black edges containing a vertex $\alpha_{j-1,i} \in \{\alpha_{j-1,i} \}$, as well as the interior of each purple edge containing a vertex $\overline{\alpha_{j-1,i}}$.
\item $G_C = \underset{j}{\cap} G^j$.
\end{enumerate}

A construction path actually always lives in the construction subgraph of its destination ltt structure.

\begin{ex}{\label{E:ConstructionSubgraph}} To find the construction subgraph $G_C$ for the ltt structure $G$ on the left (1), we remove the interior of the black edge [$\bar a, a$] to obtain the middle graph (2), then remove $a$ and the interior of all purple edges containing $a$ to obtain $G_C$ (graph (3) depicted on the right).
~\\
\vspace{-6.75mm}
\begin{figure}[H]
 \centering
 \noindent \includegraphics[width=4.5in]{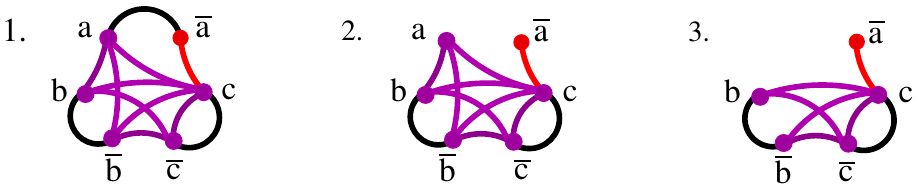}
 \label{fig:firstbuildingsubgraph1}
 \end{figure}
\end{ex}

The following lemma gives some conditions under which a path in an admissible (edge-pair)-indexed $(r;(\frac{3}{2}-r))$ ltt structure $G$ is guaranteed to be the construction path for a construction composition with destination ltt structure $G$. It also explains how to find such a construction composition.

\begin{lem}{\label{L:ConstructionAutomorsphismFromPath}} Let $G$ be an admissible $(r;(\frac{3}{2}-r))$ ltt structure and consider a smooth path
~\\
\vspace{-2mm}
$$\gamma = [d^u, \overline{x_1}, x_1, \overline{x_2}, x_2, \dots, x_{k+1}, \overline{x_{k+1}}]$$
in $G_C$ starting with $e^R$ (oriented from $d^u$ to $\overline{d^a}$) and ending with the black edge $[x_{k+1}, \overline{x_{k+1}}]$.

Edge-index $r$-petaled roses $\Gamma_{i-k-1}, \dots, \Gamma_i$ and define the homotopy equivalences
~\\
\vspace{-1mm}
$$\Gamma_{i-k-1} \xrightarrow{g_{i-k}} \Gamma_{i-k} \xrightarrow{g_{i-k+1}} \cdots \xrightarrow{g_{i-1}}\Gamma_{i-1} \xrightarrow{g_i} \Gamma_i$$
by $g_l: e_{l-1,s} \mapsto e_{l,t_l} e_{l,s}$, where $D_0(e_{l,t_l})= \overline{x_{i-l+1}}$, and $g_l(e_{l-1,j})=e_{l,j}$ for $e_{l-1,j} \neq e_{l-1,s}^{\pm 1}$.

Define the ltt structures (with respective bases $\Gamma_j$) $G_t$, for $i-k-1 \leq t \leq i$, by having: \newline
\noindent 1. each $\mathcal{PI}(G_l)$ isomorphic to $\mathcal{PI}(G_i)$ via an isomorphism preserving the vertex label second indices, \newline
\noindent 2. the second index of the label on the single red vertex in each $G_l$ be ``s'' (the same as in $G_i$), and \newline
\noindent 3. the single red edge in $G_l$ be $[d_{l,s}, \overline{d_{l,t_l}}]$.

If each $G_j$ is an admissible $(r;(\frac{3}{2}-r))$ ltt structure for $\mathcal{G}$ with base $\Gamma_j$, then $(g_{i-k}, \dots, g_i; G_{i-k-1}, \dots, G_i)$ is a purified construction composition with construction path $\gamma$. For each $i-k+1 \leq l \leq i$, the triple $(g_l, G_{l-1}, G_l)$ is the extension determined by $[\overline{x_{i-l+1}}, x_{i-l+2}]$.
\end{lem}

\begin{proof} It suffices to show: A. each $(g_l, G_{l-1}, G_l)$ is the extension determined by $[\overline{x_{i-l+1}}, x_{i-l+2}]$ (so that $(g_{i-k}, \dots, g_i; G_{i-k-1}, \dots, G_i)$ is indeed a construction composition) and B. the corresponding construction path is $[d^u_i, \overline{d^a_i}, d^a_i, \overline{d^a_{i-1}}, d^a_{i-1}, \dots, d^a_{i-k+1}, \overline{d^a_{i-k}}, d^a_{i-k}]$.

(extI) holds by our requiring each $G_j$ be an $(r;(\frac{3}{2}-r))$ ltt structure with rose base graph. The $G_l$ are $(r;(\frac{3}{2}-r))$ ltt structures for $\mathcal{PI}(G)$ by (1)-(3) in the lemma statement. This, with how we defined our notation, implies (gtIII) and (extI). The second index of the red vertex label is the same in each $G_l$ as in $G_i$, giving (ext II). To see (extIII) holds by (1), note that $[\overline{x_{i-l+1}}, x_{i-l+2}]$ is in $\mathcal{PI}(G_l)$ (it is in $G$ and $\mathcal{PI}(G) \cong \mathcal{PI}(G_l)$) and would be the determining edge for the extension. (A) is proved.

The construction path is $[d^u_i, \overline{d^a_i}, d^a_i, \overline{d^a_{i-1}}, d^a_{i-1}, \dots, d^a_{i-k+1}, \overline{d^a_{i-k}}, d^a_{i-k}]$ by Lemma \ref{L:ConstructionPathSmooth}, proving (B). \qedhere
\end{proof}

It is proved in \cite{p12a} that $(g_{i-k}, \dots, g_i; G_{i-k-1}, \dots, G_i)$ is in fact the unique construction composition with $\gamma$ as its construction path. We call $\Gamma_{i-k-1} \xrightarrow{g_{i-k}} \Gamma_{i-k} \xrightarrow{g_{i-k+1}} \cdots \xrightarrow{g_{i-1}}\Gamma_{i-1} \xrightarrow{g_i} \Gamma_i$, together with its sequence of ltt structures $G_{i-k-1} \xrightarrow{D^T(g_{i-k})} G_{i-k} \xrightarrow{D^T(g_{i-k+1})} \cdots \xrightarrow{D^T(g_{i-1})} G_{i-1} \xrightarrow{D^T(g_i)} G_i$, as in the lemma, the construction composition \emph{determined by} the path $\gamma = [d^u, \overline{x_1}, x_1, \overline{x_2}, x_2, \dots, x_{k+1}, \overline{x_{k+1}}]$.

\begin{ex}{\label{E:ConstructionAutomorphism}} In the following ltt structure, $G_i$, the numbered edges give a construction path determined by the construction automorphism $a \mapsto ab\bar{c}\bar{c}bbcb$ (the automorphism fixes all other edges).
~\\
\vspace{-6.5mm}
\begin{figure}[H]
\centering
\includegraphics[width=1.5in]{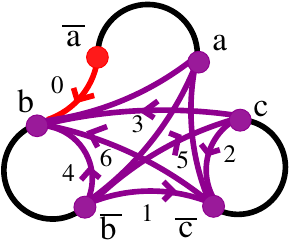}
~\\
\vspace{-4mm}
\label{fig:ConstructionPathExample}
\end{figure}

We retrieve each ltt structure $G_{i-k}$ in the construction composition by moving the red edge of $G_i$ to be attached to the terminal vertex of edge $k$ in the construction path. If the red vertex of $G_j$ is $d_s$ and the red edge is $[d_s,d_t]$, then $g_j$ is defined by $e_s \mapsto \bar{e_t}e_s$. We show the construction composition, leaving out the source ltt structure $G_{i-7}$ of the switch to highlight that it does not affect the construction path.
~\\
\vspace{-9.25mm}
\noindent \begin{figure}[H]
\centering
\noindent \includegraphics[width=6.5in]{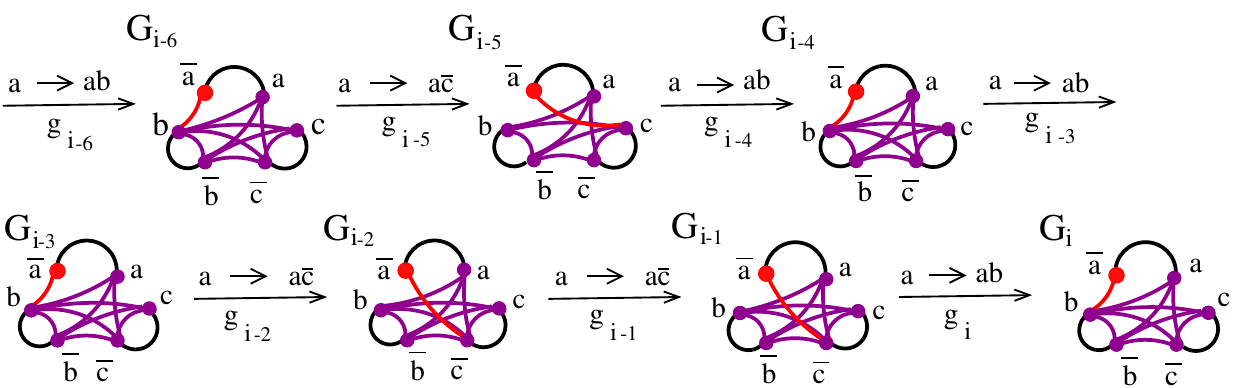}
\label{fig:ConstructionCompositionSequence}
\end{figure}
\end{ex}

\indent Lemma \ref{P:PathConstruction} is fundamental to our construction techniques. It says that construction compositions, in fact, ``build'' in the ideal Whitehead graph the images of the purple edges of the construction path:

\begin{lem}{\label{P:PathConstruction}}  Let $g$ be an $\mathcal{ID}$ type $(r; (\frac{3}{2}-r))$ representative of $\phi \in Out(F_r)$ with $\mathcal{IW}(\phi)=\mathcal{G}$. Suppose $g$ decomposes as $\Gamma = \Gamma_0 \xrightarrow{g_1} \Gamma_1 \xrightarrow{g_2} \cdots \xrightarrow{g_{n-1}}\Gamma_{n-1} \xrightarrow{g_n} \Gamma_n = \Gamma$,
with the sequence of ltt structures for $\mathcal{G}$:
$$G_{i-k-1} \xrightarrow{D^T(g_{i-k})} G_{i-k} \xrightarrow{D^T(g_{i-k+1})} \cdots \xrightarrow{D^T(g_{i-1})} G_{i-1} \xrightarrow{D^T(g_i)} G_i.$$
If $g'=g_n \circ \dots \circ g_{k+1}$ is a construction composition, then $\mathcal{G}$ contains as a subgraph the purple edges in the construction path for $g'$. \end{lem}

\begin{proof} We proceed by induction for decreasing $k$. Proof by induction is valid here since nothing in the proof will rely on $G_k$ (the only thing distinguishing $(g_k, G_k, G_{k+1})$ as a switch instead of an extension).

\indent For the base case consider $g_n \circ g_{n-1}$. By \cite{p12b} Corollary 5.6b $G_{n-1}$ has red edge $[d^u_{n-1}, \overline{d^a_{n-1}}]$. We know $g_n$ is defined by $g_n$: $e^{pu}_{n-1} \mapsto e^a_n e^u_n$ and $g_n(e_{n-1,l})= e_{n,l}$ for all $e_{n-1,l} \neq (e^{pu}_{n-1})^{\pm 1}$. Thus, since $d^{pu}_{n-1}= d^u_{n-1} \neq \overline{d^a_{n-1}}$, we know that $Dg_n(\overline{d^a_{n-1}}) = \overline{d^a_{n-1}}$. So $D^Cg_n([d^u_{n-1}, \overline{d^a_{n-1}}]) = D^Cg_n([d^{pu}_{n-1}, \overline{d^a_{n-1}}]) = [d^a_n, \overline{d^a_{n-1}}]$ and, since $D^Cg_n(\mathcal{C}(G_{n-1})) \subset \mathcal{PI}(G_n)$, $[d^a_n, \overline{d^a_{n-1}}]$ is in $\mathcal{PI}(G_n)$. The base case is proved.

\indent For the inductive step assume, for $n > s > k+1$, $G_n$ contains the purple edges of $\gamma_{g_{n,s}}$. Again by \cite{p12b} Corollary 5.6b, $e^R_{s-1}=[d^u_{s-1}, \overline{d^a_{s-1}}]$. As above, $D^Cg_s([d^u_{s-1}, \overline{d^a_{s-1}}])= [d^a_s, \overline{d^a_{s-1}} ]$ is in $\mathcal{PI}(G_s)$. Since extensions map purple edges to themselves and $D^Cg_{s}([d^u_{s}, \overline{d^a_{s-1}}])=[d^a_{s}, \overline{d^a_{s-1}}]$, $D^Cg_{n,s}([d^u_{s-1}, \overline{d^a_{s-1}}]) = D^Cg_{n,s+1}(D^tg_{s}([d^u_{s-1}, \overline{d^a_{s-1}}])) = D^Cg_{n,s+1}([d^a_s, \overline{d^a_{s-1}}])=[d^a_s, \overline{d^a_{s-1}}]$, proving the inductive step. \qedhere
\end{proof}

\subsection{Switch Paths}

We use ``switch paths'' to find switch sequences. Here switch sequences play two primary roles: ensuring our ideal decomposition actually gives a loop in $\mathcal{ID}(\mathcal{G})$ and ensuring our transition matrix is PF.

\emph{We continue with the notational abuse of the previous section (primarily ignoring second indices).}

\vskip8pt

\begin{df} (See Example \ref{E:SwitchPath}) An admissible \emph{switch sequence} for a $(r;(\frac{3}{2}-r))$ graph $\mathcal{G}$ is an admissible composition $(g_{i-k}, \dots, g_i; G_{i-k-1}, \dots, G_i)$ for $\mathcal{G}$ such that
\indent{\begin{description}
\item  [(ss1)] each  $(g_j, G_{j-1}, G_j)$ with $i-k \leq j \leq i$ is a switch and
\item  [(ss2)] $d^a_{n+1}=d^u_n \neq d^u_l=d^a_{l+1}$ and $\overline{d^a_l} \neq
d^u_n=d^a_{n+1}$ for all $i \geq n > l \geq i-k$.
\end{description}}
\indent We call the associated automorphism $g_{i,i-k}=g_i \circ \dots \circ g_{i-k}$ a \emph{switch sequence automorphism}.
\end{df}

\begin{rk} (ss2) is not implied by (ss1) and is necessary for a switch path to indeed be a path.  Certain statements in the Lemma \ref{L:SwitchPathSmooth} proof below (showing that the switch path for a switch sequence is a smooth path in the destination ltt structure) would be incorrect without (ss2).
\end{rk}

\begin{df} Let $(g_j, \dots, g_k; G_{j-1}, \dots, G_k)$ be an admissible switch sequence. Its \emph{switch path} is a path in the destination ltt structure $G_k$ traversing the red edge [$d^u_k, \overline{d^a_k}$] from its red vertex $d^u_k$ to $\overline{d^a_k}$, the black edge [$\overline{d^a_k}, d^a_k$] from $\overline{d^a_k}$ to $d^a_k$, what is the red edge [$d^u_{k-1}, \overline{d^a_{k-1}}$] = [$d^a_k, \overline{d^a_{k-1}}$] in $G_{k-1}$ (purple edge in $G_k$) from $d^a_k=d^u_{k-1}$ to $\overline{d^a_{k-1}}$, the black edge [$\overline{d^a_{k-1}}, d^a_{k-1}$] from $\overline{d^a_{k-1}}$ to $d^a_{k-1}$, continues as such through the red edges for the $G_i$ with $j \leq i \leq k$ (inserting black edges between), and ends by traversing the black edge [$\overline{d^a_{j+1}}, d^a_{j+1}$] from $\overline{d^a_{j+1}}$ to $d^a_{j+1}$, what is the red edge [$d^u_j, \overline{d^a_j}$] = [$d^a_{j+1}, \overline{d^a_j}$] in $G_j$ (purple edge in $G_k$), and then the black edge [$\overline{d^a_j}, d^a_j$] from $\overline{d^a_j}$ to $d^a_j$. In other words, a switch path alternates between the red edges (oriented from $d^u_j$ to $\overline{d^a_j}$) for the $G_j$ (for descending $j$) and the black edges between.
\end{df}

\begin{rk} We clarify here some ways in which switch paths and construction paths differ: \newline
\indent 1. The purple edges in the construction path for a construction composition $(g_{i-k}, \dots, g_i; G_{i-k-1}, \dots, G_i)$ are purple in each $G_l$ with $i-l \leq l < i$, for a switch path. They are red edges in the structure $G_l$ they are created in and then will not exist at all in the structures $G_m$ with $m<l$. The change of color (and disappearance) of red edges is the reason for (ss2) in the switch sequence definition. \newline
\indent 2. Unlike constructions paths, switch paths do not give subpaths of lamination leaves.
\end{rk}

\vskip10pt

The following lemma proves that switch paths are indeed smooth paths in destination LTT structures. It is important to note that this only holds when (ss1) and (ss2) hold.

\vskip10pt

\begin{lem}{\label{L:SwitchPathSmooth}} Let $(g_1, \dots, g_n, G_0, \dots, G_n)$ be an $\mathcal{ID}$ for a $\mathcal{G} \in \mathcal{PI}_{(r; (\frac{3}{2}-r))}$ and $(g_{i-k}, \dots, g_i; G_{i-k-1}, \dots, G_i)$ a switch sequence. Then the associated switch path forms a smooth path in the ltt structure $G_k$. \end{lem}

\begin{proof}
The red edge in $G_k$ is [$d^u_k, \overline{d^a_k}$].  We are left to show (by induction) that:

\noindent (1) For each $1 \leq l < k$, [$d^u_{l}, \overline{d^a_{l}}]=[d^a_{l+1}, \overline{d^a_{l}}$] is a purple edge of $G_k$ and

\noindent (2) alternating the purple edges [$d^a_{l+1}, \overline{d^a_{l}}$] with the black edges [$\overline{d^a_{l}}, d^a_l$] gives a smooth path in $G_k$.

We prove the base case. By the switch properties, $e^R_{k-1}$ is [$d^u_{k-1}, \overline{d^a_{k-1}}]=[d^a_k, \overline{d^a_{k-1}}$]. Since $d^a_k \neq d^u_k$ and $\overline{d^a_{k-1}} \neq d^u_k$ (by the switch sequence definition), $D^tg_k(\{d^a_k, \overline{d^a_{k-1}}\})=\{d^a_k, \overline{d^a_{k-1}}\}$. So $[d^a_k, \overline{d^a_{k-1}}$], is a purple edge in $G_k$. Since $e^R_k= [d^u_k, \overline{d^a_k}]$, by including the black edge [$\overline{d^a_k}, d^a_k$], we have a path $[d^u_k, \overline{d^a_k}, d^a_k, \overline{d^a_{k-1}} ]$ in $G_k$ (smooth, as it alternates between colored and black edges). The base case is proved.

We prove the inductive step. By the inductive hypothesis we assume the sequence of switches for $g_k,\dots ,g_{k-i}$ gives us a smooth path $[d^u_k, \dots, \overline{d^a_{k-i}}]$ in $G_k$ ending with a purple edge with ``free'' vertex $\overline{d^a_{k-i-1}}$. We know $e^R_{k-i-1}= [d^u_{k-i-1}, \overline{d ^a_{k-i-1}}]= [d^a_{k-i}, \overline{d^a_{k-i-1}}$]. As long as $d^u_l \neq d^a_{k-i}$ and $d^u_l \neq \overline{d^a_{k-i-1}}$ for $k-i \leq l \leq k$ (holding by the switch sequence definition), $D^tg_{k,k-i}(\{d^u_{(k-i-1)}, \overline{d^a_{(k-i-1)}}\})= D^tg_{k,k-i}(\{d^a_{(k-i)}, \overline{d^a_{(k-i-1)}}\}) = \{d^a_{(k-i)}, \overline{d^a_{(k-i-1)}}\}$. So, $[d^a_{(k-i)}, \overline{d^a_{(k-i-1)}}]$ is a purple edge in $G_k$.

Since $[d^u_k, \dots, \overline{d^a_{k-i}}]$ is a smooth path in $G_k$ ending with a black edge, $[d^u_k, \dots, \overline{d^a_{k-i}}, d^a_{k-i}, \overline{d^a_{k-i-1}}]$ is also a smooth path in $G_k$, as [$\overline{d^a_{k-i}},d^a_{k-i}$] is a black edge in $G_k$ and [$d^a_{k-i}, \overline{d^a_{k-i-1}}$] a purple edge in $G_k$. \qedhere
\end{proof}

\begin{ex}{\label{E:SwitchPath}} In the ltt structure $G_i$ of Example \ref{E:ConstructionAutomorphism} we number the colored edges of a switch path:
~\\
\vspace{-8.5mm}
\begin{figure}[H]
\centering
\includegraphics[width=1.2in]{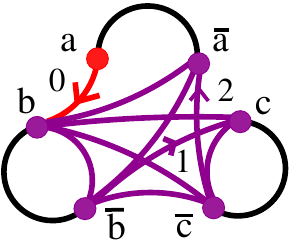}
\label{fig:SwitchPath} \\[-5mm]
\end{figure}

\noindent The switch sequence constructed from the switch path is:
~\\
\vspace{-9.75mm}
\begin{figure}[H]
\centering
\noindent \includegraphics[width=5.8in]{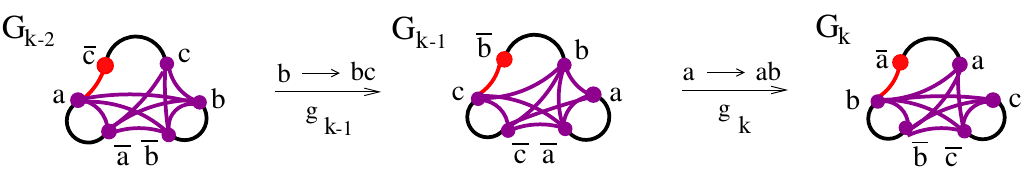}
\label{fig:SwitchSequence} \\[-5mm]
\end{figure}
\end{ex}

\noindent The red edge $e^r_k$ in $G_k$ is (0), the red edge $e^r_{k-1}$ in $G_{k-1}$ is (1), and the red edge $e^r_{k-2}$ in $G_{k-2}$ is (2).

\vskip10pt

The following lemma explains how, by inserting construction compositions into a well-chosen switch sequence, one can ensure their transition matrix is PF (for purposes of applying the FIC).

\begin{lem}{\label{L:SwitchSequence}} Suppose $g$ decomposes as $g_{i_m',i_m} \circ \cdots \circ g_{i_1',i_1}$ where:
\begin{enumerate}[itemsep=-1.5pt,parsep=3pt,topsep=3pt]
\item Each $g_{i_k',i_k}$ is a construction composition whose pure construction composition starts and ends with the same ltt structure.
\item For each edge pair $\{d_i, \overline{d_i}\}$, either $d_i$ or $\overline{d_i}$ is the red unachieved direction vertex for some $G_{i_k-1}$.
\end{enumerate}
Then the transition matrix for $g$ is PF.
\end{lem}

\begin{proof}
It suffices to show each $e^{u}_{i_k}$ is in the image of each $e^{u}_{i_j}$. In fact, it suffices to show each $e^{u}_{i_{k-1}}$ is in the image of each $e^{u}_{i_k}$. Note $g_{i_k}(e^{pu}_{(i_k-1)}) = e^a_{i_k} e^u_{i_k}$. Since $(g_{i_k}, G_{(i_k-1)}, G_{i_k})$ is a switch, $e^{pa}_{(i_k-1)}=e^{u}_{(i_k-1)}$. Since $g_{i_{(k-1)}',i_{(k-1)}}$ is a construction composition whose pure construction composition starts and ends with the same ltt structure, $e^{u}_{(i_k-1)}=e^{u}_{i_{(k-1)}}$. So $e^{pa}_{(i_k-1)}= e^{u}_{i_{(k-1)}}$ and $e^{u}_{i_k}$ maps over $e^{u}_{i_{k-1}}$. ($e^{u}_{i_{k}}$ has the same second index as $e^{pu}_{i_k}$ and $e^{u}_{i_{k-1}}$ has the same second index as $e^{a}_{i_k}$. Use Lemma 5.3 of \cite{p12b}.) \qedhere
\end{proof}

\section{Full Irreducibility Criterion}{\label{Ch:FIC}}

We prove a ``Folk Lemma'' giving a criterion, the ``Full Irreducibility Criterion (FIC),'' for an irreducible tt representative to be fully irreducible. Our original approach involved the ``Weak Attraction Theorem,'' several notions of train tracks, laminations, and the basin of attraction for a lamination. However, Michael Handel graciously provided a method to complete it, making much of our initial work unnecessary.  Our proof here uses Handel's recommendation. [K12] gives a criterion based on ours.

The proof we give uses the relative train tracks (rtts) of \cite{bh92} and the completely split relative train tracks (CTs) of \cite{fh11}. If $\phi \in Out(F_r)$ is rotationless and $\mathcal{C}$ is a nested sequence of $\phi$-invariant free factor systems, then $\phi$ is represented by a CT and filtration realizing $\mathcal{C}$ (\cite{fh11}, Theorem 4.29). We use from the CT definition (CT5): for a fixed stratum $H_t$ with unique edge $E_t$, either $E_t$ is a loop or each end of $E_t$ is in $\Gamma_{t-1}$.

Our definition of the attracting lamination for an outer automorphism will be as in \cite{bfh00}. A complete summary of relevant definitions can be found in \cite{p12a}. $\mathcal{L}(\phi)$ will denote the set of attracting laminations for $\phi$. By \cite{bfh00}, for a $\phi \in Out(F_r)$, there exists a correspondence between $\mathcal{L}(\phi)$ and the EG-strata of a rtt representative $g\colon\Gamma \to \Gamma$ of $\phi$: For each EG stratum $H_t$, there exists a unique attracting lamination $\Lambda_t$ with $H_t$ as the highest stratum crossed by the realization $\lambda \subset \Gamma$ of a $\Lambda_t$-generic line. $\Lambda(\phi)$ will denote the unique attracting lamination for an irreducible $\phi$.

Free factor support is defined in \cite{bfh00} (Corollary 2.6.5). The relevant information for our FIC proof is: if a lamination is carried by a proper free factor, then its support is a proper free factor.

\begin{prop}{\label{P:FIC}} (The Full Irreducibility Criterion)
Let $g$ be a train track representive of an outer automorphism $\phi \in Out(F_r)$ such that
~\\
\vspace{-6mm}
\begin{itemize}
\item [(I)] $g$ has no periodic Nielsen paths, \\[-6mm]
\item [(II)] the transition matrix for $g$ is Perron-Frobenius, and \\[-6mm]
\item [(III)] all local Whitehead graphs $LW(x;g)$ for $g$ are connected.
\end{itemize}
\indent Then $\phi$ is fully irreducible.
\end{prop}

\begin{proof} Suppose $g\colon \Gamma \to \Gamma$ is as in the statement. Since $g$ has a PF transition matrix, as an rtt, it has precisely one stratum and that stratum is EG. Hence, it has precisely one attracting lamination \cite{bfh00}. Since the number of attracting laminations for a tt representative is independent of the representative choice, any $\phi$ representative would also have precisely one attracting lamination.

Suppose, for contradiction's sake, $\phi$ were not fully irreducible. Then some $\phi^k$ would be reducible. If necessary, take an even higher power $R$ so that $\phi^R$ is also rotationless (this does not change reducibility). Note, since $\mathcal{L}(\phi)$ is $\phi$-invariant, any $\phi^R$ representative also has precisely one attracting lamination.

Since $\phi^R$ is reducible (and rotationless), there exists a CT representative $h: \Gamma' \to \Gamma'$ of $\phi^R$ with more than one stratum (\cite{fh11}, Theorem 4.29). Since $\phi^R$ has precisely one attracting lamination, $h$ has precisely one EG-stratum $H_t$. Each stratum $H_i$, other than $H_t$ and any zero strata, would be an NEG-stratum consisting of a single edge $E_i$ (\cite{fh11}, Lemma 4.22). We consider separately the cases where $t=1$ and $t>1$.

Since the transition submatrix for any zero stratum is zero (hence every edge of the stratum is mapped to a lower filtration element by $h$), $H_1$ could not be a zero stratum. Thus, if $t>1$, then $H_1$ is NEG and must consist of a single edge $E_1$. Since $H_1$ is bottom-most, it would be fixed, as there are no lower strata for its edge to be mapped into. According to (CT5), $E_1$ would then be an invariant loop. This would imply $\phi^R$ has a rank-1 invariant free factor. However, $g$ (hence $g^R$) was pNp-free. So $\phi^R$ could not have a rank-1 invariant free factor. We have reached a contradiction for $t>1$.

Assume $t=1$. Then $\Lambda(\phi^R)(=\Lambda(\phi))$ is carried by a proper free factor. Proposition 2.4 of \cite{bfh97} says, if a finitely generated subgroup $A \subset F_r$ carries $\Lambda_{\phi}$, then $A$ has finite index in $F_r$. The necessary conditions for this proposition are actually only: 1. the transition matrix of $g$ is irreducible and 2. each $LW(g;v)$ is connected. (Up to the contradiction in the proof of Proposition 2.4 of \cite{bfh97}, the only properties used in the proof are that the support is finitely generated, proper, and carries the lamination. The contradiction uses Lemma 2.1 in \cite{bfh97}, which shows (1) and (2) carry over to lifts of $g$ to finite-sheeted covering spaces, using no properties other than properties (1) and (2).) Assumptions (1) and (2) are assumptions in this lemma's hypotheses and $\Lambda(\phi)$ is still the attracting lamination for $g$. So we can apply the proposition to contradict $\Lambda(\phi)$ having proper free factor support: Applying the proposition, since proper free factors have infinite index, the support must be the whole group. This contradicts that the EG-stratum is $H_1$ and that there must be more than one stratum.

We have thus shown that we cannot have more than one stratum with $t=1$ or $t >1$. So all powers of $\phi$ must be irreducible and $\phi$ is fully irreducible, as desired. \qedhere
\end{proof}

\begin{lem}{\label{P:RepresentativeLoops}} Suppose $\mathcal{G} \in \mathcal{PI}_{(r; (\frac{3}{2}-r))}$ and $(g_1, \dots, g_k; G_
 0, \dots, G_k)$, $g=g_k\circ \cdots \circ g_1$, is a rotationless $\mathcal{ID}$, satisfying:
~\\
\vspace{-5mm}
\begin{description}
\item 1. $\mathcal{PI}(G(g)) \cong \mathcal{G}$. More precisely, $\mathcal{G} \cong \mathcal{SW}(g;v)$. \\[-6mm]
\item 2. And for each $1 \leq i,j \leq q$, there exists some $k \geq 1$ such that $g^k(E_j)$ contains either $E_i$ or $\bar E_i$. \\[-6mm]
\item 3. And $g$ has no periodic Nielsen paths. \\[-6mm]
\end{description}
\noindent Then $g$ is a tt representative of a $\phi \in \mathcal{AFI}_r$ such that $\mathcal{IW}(\phi)=\mathcal{G}$.
\end{lem}

\begin{proof} By the FIC, we only need to show $g$ is a tt map, the transition matrix of $g$ is PF, and $\mathcal{IW}(\phi)=\mathcal{G}$. (2) implies that the transition matrix is PF. $g$ is a tt map by \cite{p12b} Lemma 5.3a. Since $g$ has no pNp's, $\mathcal{IW}(g)=\mathcal{SW}(g)$. By the definition of $G(g)$, we know $\mathcal{PI}(G(g))=\mathcal{SW}(v;g)$. \qedhere
\end{proof}

\begin{rk} In the language of \cite{p12b}, the conditions of Lemma \ref{P:RepresentativeLoops} could be stated as
~\\
\vspace{-1mm}
$$L(g_1, \dots, g_k; G_0, G_1 \dots, G_{k-1}, G_k)=E(g_1, G_{0}, G_1) * \dots * E(g_k, G_{k-1}, G_k)$$
being a loop in $\mathcal{ID}(\mathcal{G})$ satisfying (1)-(3) and inducing a map fixing all periodic directions.
\end{rk}

\section{Nielsen path identification and prevention}{\label{Ch:NPIdentification}}

We give (Proposition \ref{P:NPIdentification}) a method for finding all ipNp's, thus pNp's, for a tt map $g\colon\Gamma \to \Gamma$ ideally decomposed as $\Gamma = \Gamma_0 \xrightarrow{g_1} \Gamma_1 \xrightarrow{g_2} \cdots \xrightarrow{g_{n-1}}\Gamma_{n-1} \xrightarrow{g_n} \Gamma_n = \Gamma$ where $(g_0, \dots, g_n; G_0, \dots, G_n)$ is an admissible composition. Note that, for each $k$, we know that $T_k= \{d^{pu}_k, d^{pa}_k \}$ is the unique illegal turn for $f_k= g_k \circ \cdots \circ g_{k+1}: \Gamma_k \to \Gamma_k$ and $\mathcal{AM}$II guarantees each $f_k$ is also a tt map.

As a warm-up, Example \ref{E:NPIdentification} demonstrates the procedure's application to an ideally decomposed representative. We then explain the procedure's steps and how we used them in Example \ref{E:NPIdentification}. This section concludes with a proof of the procedure's validity.

\begin{ex}{\label{E:NPIdentification}} We apply the procedure to show the following ideally decomposed tt map has no pNp's.

For simplicity, the ltt structures $G_i = G(f_i)$ are shown without black edges $[e_j]$ connecting the vertex pairs $\{d_j, \overline{d_j} \}$. Underneath each ltt structure $G_i$ we included the illegal turn $T_i$ for the generator $g_i$. We often abuse notation by writing $e$ for $D_0(e)$ where $e \in \{a, \bar a, b, \bar b, c, \bar c \}$.
~\\
\vspace{-9mm}
\noindent \begin{figure}[H]
\centering
\includegraphics[width=6.5in]{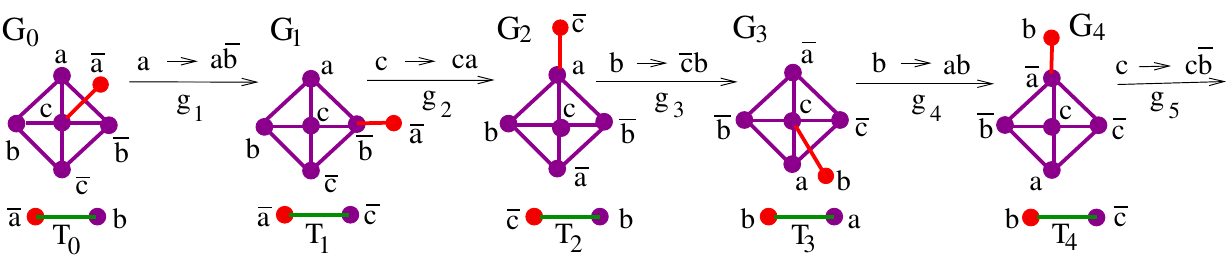}
\label{fig:NPAlgorithmExample1}
\end{figure}
~\\
\vspace{-20mm}
\noindent \begin{figure}[H]
\centering
\includegraphics[width=6.5in]{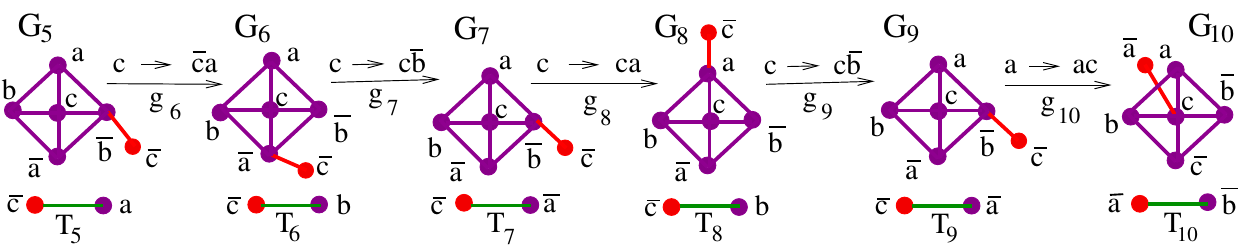}
\label{fig:NPAlgorithmExample2}
\end{figure}
~\\
\vspace{-20mm}
\noindent \begin{figure}[H]
\centering
\includegraphics[width=6.4in]{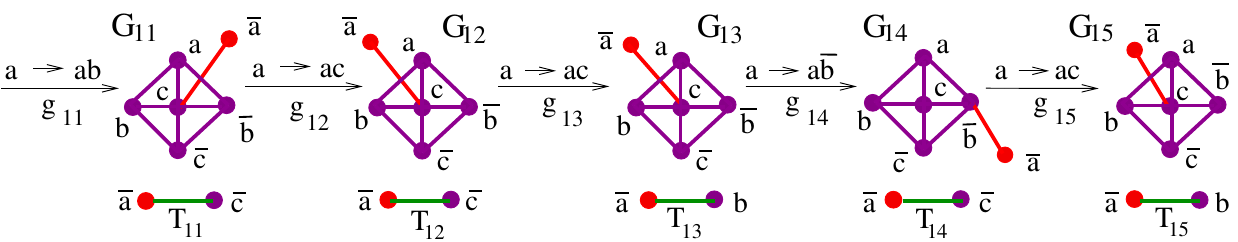}
\label{fig:NPAlgorithmExample3} \\[-5mm]
\end{figure}

\vskip10pt

\noindent \emph{Procedure Applied}:

Since $\{\overline{a}, b\}$ is the only illegal turn for $g$, any ipNp would contain it as its unique illegal turn. We are thus trying to find an ipNp $\overline{\rho_1}\rho_2$ where $\rho_1= \overline{a}e_2e_3\dots$ and $\rho_2=be_2'e_3'\dots$ are legal paths and all $e_i, e_i' \in \mathcal{E}(\Gamma)$, except that the final edge in either $\rho_1$ or $\rho_2$ may be partial.

Since $g_1(b)=b$ is the initial subpath of $g_1(\overline{a})= b\overline{a}$, proper cancelation would force $\rho_2$ to contain another edge $e_2'$ after $b$ (Proposition \ref{P:NPIdentification} I). Since $\overline{a}$ labels the red vertex in $G_1$ (i.e. $D_0(\overline{a})$ is the unachieved direction $d^u_1$), the edge $e_2'$ would be such that $D_0(e_2')$ is a preimage under $Dg_1$ of $D_0(\overline{c})$ (proper cancelation requires $Dg_{2,1}(e_2')= Dg_2(\overline{a})= D_0(\overline{c})$ but, since $Dg_1(e_2')$ cannot be the unachieved direction $\overline{a}$, and the only other preimage under $Dg_2$ of the twice-achieved direction $\overline{c}$ is the other direction in the illegal turn $T_1 = \{\overline{a}, \overline{c} \}$, namely the twice-achieved direction $\overline{c}$, we must have $Dg_1(e_2')= \overline{c}$). The only preimage under $Dg_1$ of $\overline{c}$ is $\overline{c}$. So $e_2'= \overline{c}$. Thus, so far, we have $\rho_1= \overline{a}\dots$ and $\rho_2=b\overline{c}\dots$

Since $g_{2,1}(\overline{a})=g_2(b)\overline{a}$ is the initial subpath of $g_{2,1}(b\overline{c})= g_2(b)\overline{a}\overline{c}$, we know $\rho_1$ would contain another edge $e_2$ after $\overline{a}$ (Proposition \ref{P:NPIdentification} III). Since $\overline{c}$ labels the red unachieved direction vertex in $G_2$, we know $D_0(e_2)$ cannot be a preimage under $Dg_{2,1}$ of $\overline{c}$. So our best hope is $Dg_{3,1}(e_2)= Dg_3(\overline{c})$. This can only happen if $Dg_{2,1}(e_2)$ is a preimage under $Dg_3$ of the other direction in the illegal turn $T_2 = \{\overline{c}, b \}$, namely the twice-achieved direction $b$. There are two such preimages under $Dg_{2,1}$ (the two direction in the illegal turn $T_0 = \{\overline{a}, b \}$):  Case 1 will be where $e_2=\overline{a}$ and Case 2 where $e_2=b$. We analyze each of these case.

For Case 1, suppose $e_2= \overline{a}$. Since $g_{3,1}(b\overline{c})= g_{3,2}(b)g_3(\overline{a})\overline{c}$ is the initial subpath of $g_{3,1}(\bar a \bar a)= g_{3,2}(b)g_3(\overline{a})\overline{c}b\overline{a}$, we know $\rho_2$ would contain another edge $e_3'$ after $\overline{c}$. Since $b$ labels the red unachieved direction vertex in $G_3$, we would need $D_0(e_3')$ to be a preimage of $a$ under $Dg_3$ (since $T_3 = \{a, b \}$, this follows as above). The only such preimage is $a$. So $e_3'=a$, giving $\rho_1= \bar a \bar a \dots$ and $\rho_2= b\overline{c}a \dots$.

Note that $g_{4,1}(b\overline{c}a)= g_{4,2}(b)g_{4,3}(\overline{a})g_4(\overline{c})a\overline{b}\overline{a}c$ and $g_{4,1}(\bar a \bar a)= g_{4,2}(b)g_{4,3}(\overline{a})g_4(\overline{c})ab\overline{a}$. So, since $\{\overline{b}, b\} \neq T_4$ (and $b \neq \overline{b}$), we could not have $\rho_1= \bar a \bar a$ and $\rho_2= b\overline{c}a$ (Proposition \ref{P:NPIdentification} IIc). Case 1 could not occur.

For Case 2, suppose $e_2=b$. Since $g_{3,1}(b\overline{c})= g_{3,2}(b)g_3(\overline{a})\overline{c}$ is the initial subpath of $g_{3,1}(\overline{a}b)= g_{3,2}(b)g_3(\overline{a})\overline{c}b$, we know $\rho_2$ would contain another edge $e_3'$ after $\overline{c}$. Since $b$ labels the red unachieved direction vertex in $G_3$, we can follow the logic above and see that $e_3'$ would be $a$, giving $\rho_1= \overline{a}b \dots$ and $\rho_2= b\overline{c}a \dots$. Since $g_{4,1}(b\overline{c}a)= g_{4,2}(b)g_{4,3}(\overline{a})g_4(\overline{c})a\overline{b}\overline{a}c$ and $g_{4,1}(\overline{a}b)= g_{4,2}(b) g_{4,3}(\overline{a}) g_4(\overline{c})ab$, cancelation leaves $\{\overline{b}, b\}$. As above, we could not have $\rho_1=\overline{a}b \dots$ and $\rho_2= b\overline{c}a \dots$.

This rules out all possibilities for $\overline{\rho_1}\rho_2$ and so $g$ has no ipNp's, thus no pNp's, as desired. \qedhere
\end{ex}

\vskip10pt

\noindent \emph{Explanation of Procedure Applied}:

\vskip5pt

\indent \emph{Let $(g_0, \dots, g_n; G_0, \dots, G_n)$ be an $\mathcal{ID}$ admissible composition with the standard notation. We give the general procedure for finding any ipNp $\rho= \overline{\rho_1} \rho_2$ for $g= g_0 \circ \cdots \circ g_n$, where $\rho_1 = e_1\dots e_m$; $\rho_2 = e_1'\dots e_{m'}'$; $e_1,\dots, e_m, e_1',\dots,e_{m'}' \in \mathcal{E}(\Gamma)$; and $\{D_0(e_1), D_0(e_1')\}$ $=\{d_1, d_1'\}$ is the unique illegal turn of $\rho$. We let $\rho_{1,k}=e_1\dots e_k$ and $\rho_{2,l}= e_1'\dots e_l'$ throughout the procedure. After each step is explained in italics, we explain its use in Example \ref{E:NPIdentification}.}

\begin{description}
\item[(I)] \emph{Apply $g_1$, $g_2$, etc, to $e_1$ and $e_1'$ until $Dg_{j,1}(e_1')=Dg_{j,1}(e_1)$. Either $g_{j,1}(e_1)$ is the initial subpath of $g_{j,1}(e_1')$ or vice versa. Without generality loss (or adjust notation) assume $g_{j,1}(e_1')$ is the subpath of $g_{j,1}(e_1)$, so $g_{j,1}(e_1)= g_{j,1}(e_1')t_2 \dots$, for some edge $t_2$. Then $\rho_2$ must contain another edge $e_2'$.} \\[-6.5mm]
\end{description}

\hspace{.45in}\parbox{5.75in}{\fontsize{10pt}{12pt}\selectfont $g_1(b)=b$ was the initial subpath of $g_1(\overline{a})= b\overline{a}$. This implied $\rho_2$ had an edge after $b$.}
~\\
\vspace{-6mm}
\begin{description}
\item[(II)] \emph{Continue composing generators $g_i$ until either:} \\[-8mm]
\end{description}
~\\
\vspace{-9mm}
\begin{itemize}
\item[\textbf{(a)}] \emph{one obtains a $g^{p'}$ with  $g^{p'}(\rho_{2,k})=\tau'e_1'\dots$ and $g^{p'}(\rho_{1,s})=\tau'e_1\dots$ for a legal path $\tau'$ (proceed to V),} \\[-4.5mm]
\end{itemize}

\hspace{.45in}\parbox{5.75in}{\fontsize{10pt}{12pt}\selectfont IIa did not occur in Example \ref{E:NPIdentification}. When it occurs, it makes an ipNp promising. Va may identify an ipNp, if it exists, as does IVc. However, IVc involves ``trimming,'' since it involves the case where the  path's initial and final edges are only partial edges. IVb and IVd direct one to possibly find an ipNp by continuing to add edges.}
~\\
\vspace{-2.5mm}
\indent \begin{itemize}
\item[\textbf{(b)}]  \emph{$g_{j,l}(\rho_{2,k})$ is a subpath of $g_{j,l}(\rho_{1,s})$ or vice versa (proceed to III),} \\[-4mm]
\end{itemize}

\hspace{.45in}\parbox{5.75in}{\fontsize{10pt}{12pt}\selectfont IIb occurs in both Case 1 and Case 2 of Example \ref{E:NPIdentification}. In Case 1, IIIa is used to obtain $\rho_{1,2}=\bar a \bar a$ and $\rho_{2,3}= b \bar c a$. In Case 2, IIIa yields $\rho_{1,2}=\bar a b$ and $\rho_{2,3}= b \bar c a$.}
~\\
\vspace{-3.25mm}
\indent \begin{itemize}
\item[\textbf{(c)}]  \emph{or some $g_{l,1}\circ g^{p'}(\rho_{2,k})= \tau'\gamma_{2,k}$ and $g_{l,1}\circ g^{p'}(\rho_{1,s})= \tau'\gamma_{1,s}$ where $T_l \neq \{D_0(\gamma_{2,k}), D_0(\gamma_{1,s})\}$. In this case there cannot be an ipNp with $\overline{\rho_{2,k}}\rho_{1,s}$ as a subpath. Proceed to VI.} \\[-4.5mm]
\end{itemize}

\hspace{.45in}\parbox{5.75in}{\fontsize{10pt}{12pt}\selectfont In both Case 1 and Case 2, after applying IIIa to obtain $\rho_{1,2}$ and $\rho_{2,3}$, IIc occurs with $\{D_0(\overline{b}), D_0(b)\} \neq T_4$. In both cases, $\tau=g_{4,2}(b)g_{4,3}(\overline{a})g_4(\overline{c})a$. In Case 1, $\{D_0(\gamma_{2,k}), D_0(\gamma_{1,s})\}=\{D_0(\overline{b}\overline{a}c), D_0(b\overline{a})\}= \{D_0(\overline{b}), D_0(b)\}$. In Case 2, $\{D_0(\gamma_{2,k}), D_0(\gamma_{1,s})\}=\{D_0(\overline{b}\overline{a}c), D_0(b)\}=\{D_0(\overline{b}), D_0(b)\}$.}
~\\
\vspace{-2.5mm}
\begin{description}
\item[(III)] \emph{Suppose $g_{j,1}(\rho_{1,k})=g_{j,1}(\rho_{2,s})t_{s+1}\dots$ (or switch $e_i$ for $e_i'$, $\rho_1$ for $\rho_2$, etc). $\rho_2$ must have another edge $e_{s+1}'$ after $\rho_{2,s}$. There are two cases to consider:} \\[-4.5mm]
\end{description}

\hspace{.45in}\parbox{5.75in}{\fontsize{10pt}{12pt}\selectfont Since $\rho_{1,1}=b$, $\rho_{2,1}=\overline{a}$, and $g_{2,1}(\overline{a})=g_2(b)\overline{a}$ is the initial subpath of $g_{2,1}(b\overline{c})= g_2(b)\overline{a}\overline{c}$, we assumed $g_{2,1}(\rho_{1,1})=g_{2,1}(\rho_{2,1})t_2 \dots$. In particular, $t_{s+1}=\overline{a}$.}
~\\
\vspace{-3.5mm}
\begin{itemize}
\item[\textbf{(a)}] \emph{If $D_0(t_{s+1})=d^{u}_{j}$, then the different possibilities for $e_{s+1}'$ are determined by the directions $d_{s+1}'$ such that $T_j=\{Dg_{j,1}(d_{s+1}'), D_0(t_{s+1})\}$ where $D_0(e_{s+1}')=d_{s+1}'$.} \\[-4.5mm]
\end{itemize}

\hspace{.45in}\parbox{5.75in}{\fontsize{10pt}{12pt}\selectfont $D_0(\overline{a})$ labeled the red vertex in $G_1$, implying $D_0(\overline{a})=d^{u}_1$, i.e. $D_0(t_{s+1})=d^{u}_{j}$. $T_1= \{\overline{a}, \overline{c} \}$ implied $Dg_1(d_2')= \overline{c}$. Since the only preimage of $\overline{c}$ under $Dg_1$ was $\overline{c}$, we needed $d_2'= D_0(\overline{c})$, $e_2'= \overline{c}$, and $\rho_{2,2}= b\overline{c}$. We hit this case again when determining possibilities for $e_2$.}
~\\
\vspace{-2.5mm}
\indent \begin{itemize}
\item[\textbf{(b)}] \emph{If $D_0(t_{s+1}) \neq d^{u}_{j}$, the possibilities for $e_{s+1}'$ are all edges $e_{s+1}'$ such that $Dg_{j,1}(D_0(e_{s+1}'))=D_0(t_{s+1})$.} \\[-4.5mm]
\end{itemize}

\hspace{.45in}\parbox{5.75in}{\fontsize{10pt}{12pt}\selectfont IIIb did not occur in Example \ref{E:NPIdentification}. If $D_0(\overline{a}) \neq d^u_1$, we would have looked for edges $e_2'$ with $Dg_1(d_2')=\overline{a}$. (Otherwise we could not have had $Dg_{2,1}(d_2')=\overline{a}$.)}

\vskip3pt

\noindent \indent \emph{Since $\rho_2$ must be legal, ignore choices for $d_{s+1}'$ where $T_0= \{\overline{d_s'}, d_{s+1}'\}$ is the illegal turn for $g$. Each remaining $d_{s+1}'$ in (a) or (b) gives another prospective ipNp one must apply the steps to.}

\vskip4.5pt

\hspace{.45in}\parbox{5.75in}{\fontsize{10pt}{12pt}\selectfont Only $e_2'=\overline{c}$ satisfied $Dg_1(e_2')= \overline{c}$. However, both $\overline{a}$ and $b$ (referred to as Case 1 and Case 2) gave prospective directions analyzed in finding $e_2$. For Case 1 and Case 2 we separately continued through the steps.}
~\\
\vspace{-2.5mm}
\begin{description}
\item[(IV)] \emph{For each $p' \geq 1$ with $g^{p'}(\rho_{2,m})=\tau'e_1\dots$ and $g^{p'}(\rho_{1,n})=\tau'e_1'\dots$ for some legal path $\tau'$ (and appropriate m and n), check if $g^{p'}_{\#}(\overline{\rho_{1,n}}\rho_{2,m}) \subset \overline{\rho_{1,n}}\rho_{2,m}$ or vice versa.} \\[-5.5mm]
\indent \begin{itemize}
\item[\textbf{(a)}] \emph{If, for some $p' \geq 1$, $g^{p'}_{\#}(\overline{\rho_{1,n}}\rho_{2,m}) = \overline{\rho_{1,n}}\rho_{2,m}$, then $\overline{\rho_{1,n}}\rho_{2,m}$ is the only possible ipNp for $g$.}\\[-5mm]
\item[\textbf{(b)}] \emph{For each $p' \geq 1$ such that $g^{p'}_{\#}(\overline{\rho_{1,n}}\rho_{2,m}) \subset \overline{\rho_{1,n}}\rho_{2,m}$ (containment proper), proceed to V.} \\[-5mm]
\item[\textbf{(c)}] \emph{If $\overline{\rho_{1,n}}\rho_{2,m} \subset g^{p'}_{\#}(\overline{\rho_{1,n}}\rho_{2,m})$ (containment proper): The final occurrence of $e_n$ in the copy of $\rho_{1,n}$ in $g^{p'}(\overline{\rho_{1,n}}\rho_{2,m})$ must be from $g^{p'}(e_n)$ and the final occurrence of $e_m'$ in the copy of $\rho_{2,m}$ in $g^{p'}(\overline{\rho_{1,n}}\rho_{2,m})$ must be from $g^{p'}(e_m')$. Thus, $e_n$ and $e_m'$ have fixed points. Replace $\overline{\rho_{1,n}}\rho_{2,m}$ with $\overline{\rho_{1,n}'} \rho_{2,m}'$, where $\overline{ \rho_{1,n}'}\rho_{2,m}'$ is $\overline{\rho_{1,n}}\rho_{2,m}$, but with $e_n$ and $e_m'$ replaced by partial edges ending at the fixed points. Repeat until some $\overline{ \rho_{1,n}'}\rho_{2,m}'$ is an ipNp.} \\[-5mm]
\item[\textbf{(d)}] \emph{If we do not have $g^{p'}_{\#}(\overline{\rho_{1,n}}\rho_{2,m}) \subset \overline{\rho_{1,n}}\rho_{2,m}'$ or vice versa for any $1 \leq p' \leq b$, there is only one possibility for $\overline{\rho_{2,m}}\rho_{1,n}$ to be a subpath of an ipNp. It is when $g^{p'}_{\#}(\overline{\rho_{1,n}}\rho_{2,m}) =\overline{\gamma_{1,n}}\gamma_{2,m}$ and either $\gamma_{1,n} \subset \rho_{1,n}$ and $\rho_{2,m} \subset \gamma_{2,m}$ or $\gamma_{2,m} \subset \rho_{2,m}$ and $\rho_{1,n} \subset \gamma_{1,n}$. In this case, apply V to the side too short. Otherwise, there cannot be an ipNp with $\overline{\rho_{2,m}}\rho_{1,n}$ as a subpath (proceed to VI).} \\[-7.5mm]
\end{itemize}
\end{description}
~\\
\vspace{-9mm}
\begin{description}
\item[(VI)] \emph{Assume $g^{p'}(\bar \rho_{1,n}\rho_{2,m}) \subset \overline{\rho_{1,n}}\rho_{2,m}$ (containment proper) and, without generality loss, $g^{p'}(\overline{\rho_{1,n}}\rho_{2,m})t_{m+1} \subset \overline{\rho_{1,n}}\rho_{2,m}$ for some $t_{m+1}$. For each $e_i$ such that $Dg^{p'}(D_0(e_i))=D_0(t_{m+1})$ and $\{D_0(\overline{e_{i-1}}), D_0(e_i) \} \neq \{D_0(e_1), D_0(e_1') \}$ (the illegal turn for $g$), return to V with $\rho_{2,m+1}$, where $e_{m+1}=e_i$.} \\[-7.5mm]
\end{description}
~\\
\vspace{-9mm}
\begin{description}
\item[(VI)] \emph{Rule out the other possible subpaths that arose via this procedure (by different choices of $d_i$, as in III or V). If there are no other possible subpaths, there are no ipNp's (thus no pNp's) for $g$.} \\[-4.5mm]
\end{description}

\hspace{.45in}\parbox{5.75in}{\fontsize{10pt}{12pt}\selectfont In discovering options for $e_2$ (after analyzing Case 1), VI sent us back to Case 2.} \newline

\vskip8pt

\begin{prop}{\label{P:NPIdentification}} Let $(g_0, \dots, g_n; G_0, \dots, G_n)$ be an $\mathcal{ID}$ admissible composition with notation as above. Then the procedure described in steps (I)-(VI) determines all ipNp's for $g= g_0 \circ \cdots \circ g_n$.
\end{prop}

\noindent We use the following lemma in the proposition proof (see \cite{p12a} for proofs).

\begin{lem} \cite{p12a}
~\\
\vspace{-6.25mm}
\begin{description}
\item[a.] Subpaths of legal paths are legal. \\[-6.5mm]
\item[b.] For train tracks, images of legal paths and turns are legal. \\[-5mm]
\end{description}
\end{lem}

\noindent We also remind the reader:
~\\
\vspace{-6.25mm}
\begin{enumerate}
\item Since $d^u_k$ will be one vertex of the illegal turn $T_k$ of $G_k$, $T_k$ cannot also be a purple edge in $G_k$. \\[-6mm]
\item $d^u_k$ must be a vertex of the red edge $[t^R_k]$ of $G_k$. \\[-6mm]
\item For each $k$, $T_k$ is not the red edge in $G_k$, so is not represented by any edge in $G_k$. \\[-6mm]
\end{enumerate}

\vskip3pt

\begin{proof} \emph{(of Proposition)} We begin with an argument used repeatedly. Since $\rho=\overline{\rho_1} \rho_2$ is an ipNp, $\rho_1$ and $\rho_2$ are both legal. Since subpaths of legal paths are legal and since the $g_{k,1}$ images of legal paths are legal, $g_{k,1}(e_1 \dots e_l)$ and $g_{k,1}(e_1' \dots e'_{l'})$ are legal paths for each $1 \leq k \leq n$, $1 \leq l \leq m$, and $1 \leq l' \leq n'$.

Since $\{D_0(e_1), D_0(e_1')\}$ is illegal, $Dg_{j,1}(e_1')=Dg_{j,1}(e_1)$, for some $j$. We show either $g_{j,1}(e_1)$ is the initial subpath of $g_{l,1}(e_1')$ or vice versa. Let $d_1=D_0(e_1)$ and $d_1'=D_0(e_1')$. Since $\{D_0(e_1), D_0(e_1')\}$ is illegal for $g$ and $T_0=\{d^{pu}_0, d^{pa}_0\}$ is $g$'s only illegal turn, $\{d_1, d_1'\}=\{d^{pu}_0, d^{pa}_0\}$.  Without generality loss suppose $d_1=d^{pu}_0$ (and $d_1'=d^{pa}_0$). Since $g_1(e^{pu}_0) = e^a_1 e^u_1$, we have $g_1(e_1)= g_1(e^{pu}_0)= e^a_1e^u_1$ and $g_1(e_1')=g_1(e^{pa}_0)=e^a_1$. So $g_1(e_1')$ is the initial subpath of $g_1(e_1)$. Since $g_{j,2}$ is an automorphism, $g_{j,2}(g_1(e_1'))=g_{j,2}(e^a_1)$ is the initial subpath of $g_{j,2}(g_1(e_1)) = g_{j,2}(e^a_1 e^u_1) = g_{j,2}(e^a_1) g_{j,2}(e^u_1)$, as desired. Left to show for I is that $\rho_2$ must contain a second edge $e_2'$. Suppose it did not. Then tightening would cancel all of $g_{j,1}(\rho_2)$ with an initial subpath of $g_{j,1}(\rho_1)$, so would cancel all of $g_{j,1}(\rho_2)$ with an initial subpath of $g_{j,1}(\rho_1)$. Thus, $(g_{j,1})_{\#}(\rho)= (g_{j,1})_{\#}(\overline{\rho_1}\rho_2)$ would be a subpath of $g_{j,1}(\rho_2)$, hence would be legal, as would $g^p_{\#}(\rho)=(g^{p-1} \circ g_{n,j+1})_{\#}((g_{j,1})_{\#}(\rho))$ for all $p$. This contradicts that some $g^p_{\#}(\rho)=\rho$, which has an illegal turn.

We next show applying the $g_i$ results in IIa, IIb, or IIc. For each $l$, $k$, and $s$, $g_{l,1}(\rho_{2,k})$ and $g_{l,1}(\rho_{1,s})$ are both legal. Write $g_{l,1}(\rho_{2,k})= \tau'\gamma_{2,k}$ and $g_{l,1}(\rho_{1,s})= \tau'\gamma_{1,s}$ where $D_0(\gamma_{2,k}) \neq D_0(\gamma_{1,s})$. If $\{D_0(\gamma_{2,k}), D_0(\gamma_{1,s})\}$ is legal, we are in IIc. If it is illegal and we are not in IIb, we can continue to apply the $g_i$ until we are in IIb, are in IIc, or reach a $g^p$. Since the only illegal turn for any $g^p$ is $\{d_1, d_1'\}$, we would by necessity be in IIa.

We show for IIc: if $g_{l,1}\circ g^{p'}(\rho_{2,k}) = \tau' \gamma_{2,k}$ and $g_{l,1}\circ g^{p'}(\rho_{1,s})=\tau'\gamma_{1,s}$ where $\{D_0(\gamma_{1,s}), D_0(\gamma_{2,k})\}$ is legal for $g_l$, then adding edges to $\rho_{2,k}$ and $\rho_{1,s}$ cannot give an ipNp for $g$. Now $g_{l,1}\circ g^{p'}(\rho_2)=\tau' \gamma_2$ and $g_{l,1}\circ g^{p'}(\rho_{1})=\tau'\gamma_{1}$ where $\gamma_{2,k}$ is an initial subpath of $\gamma_2$ (both legal) and $\gamma_{1,s}$ is an initial subpath of $\gamma_{1}$ (both legal). Since $\{D_0(\gamma_{1}), D_0(\gamma_{2})\} = \{D_0(\gamma_{1,s}), D_0(\gamma_{2,k})\}$ is legal, $(g_{l,1}\circ g^{p'})_{\#}(\rho) = (g_{l,1}\circ g^{p'})_{\#}(\overline{\rho_{1}} \rho_2) = \bar\gamma_1\gamma_2$, which is a legal path. Let $p$ be such that $g^p_{\#}(\rho)=\rho$. (Without generality loss assume $p > p'$, by replacing $p$ by an integer multiple of $p$ if necessary). Then, $g^p_{\#}(\rho)= ((g^{p-p'-1} \circ g_{n, l+1}) \circ (g_{l,1} \circ g^{p'}))_{\#}(\rho)=  (g^{p-p'-1} \circ g_{n, l+1})_{\#} ((g_{l,1} \circ g^{p'})_{\#}(\rho))= (g^{p-p'-1} \circ g_{n,l})_{\#}(\overline{\gamma_1}\gamma_2)= (g^{p-p'-1} \circ g_{n,l})(\overline{\gamma_1}\gamma_2)$, since $\overline{\gamma_1}\gamma_2$ is a legal path. So $g^p_{\#}(\rho)$ is legal, as images under admissible compositions of legal paths are legal. This contradicts that $g^p_{\#}(\rho)=\rho$ is not legal. We have verified all in II needing verification.

We prove the claims of IIIa and IIIb: $g_{j,1}(\rho_{1,k})= g_{j,1}(\rho_{2,s})t_{s+1} \dots$ implies $(g_{j,1})_{\#}(\overline{\rho_{1,k}} \rho_{2,s}) = \dots \overline{t_{s+1}} = \overline{\gamma}$ for a legal path $\gamma$ (legal, since a subpath of a legal path). Since $e_{s+1}'$ is a legal path, $g_{j,1}(e_{s+1}')$ is also. Thus, $(g_{j,1})_{\#}(\overline{\rho_{1,k}} \rho_{2,s}) = (\overline{\gamma} g_{j,1}(e_{s+1}'))_{\#}$, which is $\overline{\gamma} g_{j,1}(e_{s+1}')$ unless $Dg_{j,1}(d_{s+1}')= D_0(t_{s+1})$. It is then legal (causing a contradiction) unless  $\{D_0(\gamma), D_0(g_{j,1}(e_{s+1}')) \}$ is illegal, i.e. $T_j=\{Dg_{j,1}(D_0(e_{s+1}')),D_0( t_{s+1}) \}$.

Suppose, as in IIIa, that $D_0(t_{s+1})=d^u_j$. Then, $D_0(t_{s+1})$ is not in the image of $Dg_j$, thus is not in the image of $Dg_{j,1}=D(g_j \circ g_{j-1,1})$. So $Dg_{j,1}(d_{s+1}') \neq D_0( t_{s+1})$. This implies $(\overline{\gamma} g_{j,1}(e_{s+1}'))_{\#} = \overline{\gamma} g_{j,1}(e_{s+1}')$, which will be a legal path unless $T_j=\{Dg_{j,1}(d_{s+1}'),D_0( t_{s+1})\}$. However, if $(g_{j,1})_{\#}(\overline{\rho_{1,k}} \rho_{2,s}$) $= \overline{\gamma} g_{j,1}(e_{s+1}')$, then $(g_{j,1})_{\#}(\rho) = (g_{j,1})_{\#}(\overline{\rho_1} \rho_2)= (g_{j,1})_{\#} (\overline{e_m} \dots \overline{e_{k+1}})\overline{\gamma}g_{j,1}(e_{s+1}') (g_{j,1})_{\#}(e_{s+2}'\dots e_m')=$ \newline \noindent $g_{j,1}(\overline{e_m} \dots \overline{e_{k+1}}) \overline{\gamma} g_{j,1}(e_{s+1}') g_{j,1}(e_{s+2}' \dots e_m')$, since $\rho_1$ and $\rho_2$ are legal and edge images are legal. But $g_{j,1}(\overline{e_m}\dots\overline{e_{k+1}}) \overline{\gamma}$ is a subpath of $g_{j,1}(\overline{\rho_{1}})$, so is legal, and $g_{j,1}(e_{s+1}') g_{j,1}(e_{s+2}' \dots e_m')$ is a subpath of $g_{j,1}(\rho_2)$, so is legal, and we still have $\overline{\gamma} g_{j,1}(e_{s+1}')$ is legal, together making $(g_{j,1})_{\#}(\rho)$ legal. This contradicts that some $g^p_{\#}(\rho)$ $= (g^{p-1}\circ g_{j,n+1})_{\#} (g_{j,1})_{\#}(\rho))$ must be $\rho$, which has an illegal turn. So, $T_j = \{Dg_{j,1}(d_{s+1}'), D_0(t_{s+1})\}$, as desired.

Suppose, as in IIIb, $D_0(t_{s+1}) \neq d^u_j$. For contradiction's sake suppose $Dg_{j,1}(d_{s+1}') \neq D_0(t_{s+1})$, where $D_0(e_{s+1}') = d_{s+1}'$. Again $(\overline{\gamma} g_{j,1}(e_{s+1}'))_{\#} = \overline{\gamma} g_{j,1}(e_{s+1}')$. Also, $Dg_{j,1}(d_{s+1}')$ cannot be $d^u_j$ (as $d^u_j \notin \mathcal{IM}(Dg_j)$) and $D_0(t_{s+1}) \neq d^u_j$, implying $T_j \neq \{Dg_{j,1}(d_{s+1}'),D_0( t_{s+1}) \}$. Then, since $Dg_{j,1}(d_{s+1}') \neq D_0(t_{s+1})$, that $\overline{\gamma} g_{j,1}(e_{s+1}')$ is legal, causes a contradiction as above. So $Dg_{j,1}(d_{s+1}') = D_0(t_{s+1})$, as desired. Finally observe that choices for $e_{s+1}'$ with $T_0= \{D_0(e_{s}'), D_0(e_{s+1}')\}$ must be thrown out, as $\rho_2$ must be legal.

As in IV, suppose $g_{l,1} \circ  g^{p'}(\rho_{2,m})=\tau'e_1'\dots$ and $g_{l,1} \circ  g^{p'}(\rho_{1,n}) = \tau' e_1\dots$ for a legal path $\tau'$ (and appropriate $m$ and $n$). It is clear these are the only options. Also, Va is true by definition, Vb refers us to a later step, and IVc is fairly straight-forward. So we focus on IVd. We first need that there is only possibility for $\overline{\rho_{1,n}}\rho_{2,m}$ to be a subpath of an ipNp. Suppose, for no power $p'$, is $g^{p'}_{\#}(\overline{\rho_{1,n}}\rho_{2,m}) = \overline{\rho_{1,n}}\rho_{2,m}$ or $g^{p'}_{\#}(\overline{\rho_{1,n}}\rho_{2,m}) \subset \overline{\rho_{1,n}}\rho_{2,m}$ or $\overline{\rho_{1,n}}\rho_{2,m} \subset g^{p'}_{\#}(\overline{\rho_{1,n}}\rho_{2,m})$ or $g^{p'}_{\#}(\overline{\rho_{1,n}}\rho_{2,m}) =\overline{\gamma_{1,n}}\gamma_{2,m}$ where either $\gamma_{1,n} \subset \rho_{1,n}$ and $\rho_{2,m} \subset \gamma_{2,m}$ or $\gamma_{2,m} \subset \rho_{2,m}$ and $\rho_{1,n} \subset \gamma_{1,n}$. For contradiction's sake, suppose some $\overline{\rho_{1,n+k}}\rho_{2,m+l}$ containing $\overline{\rho_{1,n}}\rho_{2,m}$ is a period-$p$ ipNp. Since $\overline{\rho_{1,n+k}}\rho_{2,m+l}$ is an ipNp, $\rho_{1,n+k}$ and $\rho_{2,m+l}$ are both legal (as are $\rho_{1,n}$ and $\rho_{2,m}$). Thus, $g^{p}_{\#}(\overline{\rho_{1,n+k}}\rho_{2,m+l})= g^{p}(\overline{e_{n+k}'} \dots \overline{e_{n+1}})g^{p'}_{\#}(\overline{\rho_{1,n}}\rho_{2,m})g^{p}(e_{m+1}\dots e_{m+l})$. So $g^{p'}_{\#}(\overline{\rho_{1,n}}\rho_{2,m}) \subset g^{p}_{\#}(\overline{\rho_{1,n+k}}\rho_{2,m+l}) = \overline{\rho_{1,n+k}}\rho_{2,m+l}$. We are in a scenario we said could not occur, a contradiction.

There is nothing to prove in V since the conditions for IV still hold. \qedhere
\end{proof}

\begin{df}
An admissible composition $(g_1, \dots, g_k; G_0, \dots, G_k)$ will be a \emph{Nielsen path prevention sequence} if no $\mathcal{ID}$ map $(g_1', \dots, g_{k+m}'; G_0', \dots, G_{k+m}')$, with $(g_i; G_{i-1}, G_i) \sim (g_i'; G_{i-1}', G_i')$ for each $1 \leq i \leq k$, can have a pNp.
\end{df}

\begin{lem}{\label{L:NPKilling}} $G_0 \xrightarrow{g_1: z \mapsto xz} G_1 \xrightarrow{g_2: w \mapsto zw} G_2 \xrightarrow{g_3: y \mapsto y \bar w} G_3 \xrightarrow{g_4: y \mapsto y \bar x} G_4 \xrightarrow{g_5: y \mapsto y \bar w} G_5 \xrightarrow{g_6: y \mapsto y \bar x} G_6$
is a Nielsen path prevention sequence where each $\mathcal{PI}(G_i)$ is the complete ($2r-1$)-vertex graph and, listed red vertex first:
$e^R_0=[x, y]$,
$e^R_1=[z, \bar x]$,
$e^R_2=[w, \bar z]$,
$e^R_3=[\bar y, \bar w]$,
$e^R_4=[\bar y, \bar x]$,
$e^R_5=[\bar y, \bar w]$,
$e^R_6=[\bar y, \bar x]$.
\end{lem}

\begin{proof}
Suppose $\rho$ were an ipNp for $g$. Then $\rho$ would have to contain the illegal turn $\{x, z\}$ and could be written $\overline{\rho_1}\rho_2$ where $\rho_1=xe_2e_3\dots$ and $\rho_2=ze_2'e_3'\dots$ are legal. Now, $g_1(x)=x$ and $g_1(z)=xz$. So $\rho_1$ would contain an additional edge $e_2$. Also, $dg_1(e_2)=w$. So $e_2=w$. Since $g_{2,1}(xw)=xzw$ and $g_{2,1}(z)=xz$, we know $\rho_2$ would contain an additional edge $e_2'$ with $dg_{2,1}(e_2')=\bar y$. The only option is $e_2'=\bar y$. Since $g_{3,1}(xw)=xzw$ and $g_{3,1}(z\bar y)=xzw\bar y$, also $\rho_1$ would contain an additional edge $e_3$ with $dg_{3,1}(e_3)=x$. So either $e_3=x$ or $e_3=z$. Suppose $e_3=x$. Since $g_{4,1}(xwx)=xzwx$ and $g_{4,1}(z\bar y)=xzwx\bar y$, we know $\rho_1$ would contain an edge $e_4$ with $dg_{4,1}(e_4)=w$. So $e_4=\bar y$. Note that $g_{5,1}(xwx\bar y)=xzwxxw\bar y$ and $g_{5,1}(z\bar y)=xzwxw\bar y$ both start with $xzwx$ and $\{x, w\}$ is not the illegal turn for $g_6$. So we must return to the case of $e_3=z$. Since $g_{4,1}(xwz)=xzwxz$ and $g_{4,1}(z\bar y)=xzwx\bar y$ and $\{\bar y, z\}$ is not the illegal turn for $g_5$, we have exhausted all possibilities for $\rho$. Thus, $g$ was not an ipNp. \qedhere
\end{proof}

\section{Achieving all ($2r-1$)-vertex complete graphs}{\label{Ch:Achievable}}

Denote by $\mathcal{C}_r$ the complete ($2r-1$)-vertex graph. We construct (Theorem \ref{T:MainTheorem}), for each $r \geq 3$, a $\phi \in \mathcal{AFI}_r$, such that $\mathcal{C}_r \cong \mathcal{IW}(\phi)$. First we establish notation and prove a lemma.

For each $r \geq 3$ and $1 \leq j \leq 2r-1$, color $\mathcal{C}_r$ purple and let $\mathcal{G}_j$ be the graph formed from $\mathcal{C}_r$ by: \newline
\noindent 1. index edge-pair labeling $\mathcal{C}_r$'s vertices with $\{x_1, x_2, \dots, x_{2r-1} \}$, writing $X_i$ for $x_{2i-1}$ and $\overline{X_i}$ for $x_{2i}$;  \newline
\noindent 2. adding a black edge $[x_i, \overline{x_i}]$ for each $1 \leq i \leq 2r-1$ (leaving vertices $x_i$ and $\overline{x_i}$ purple); and \newline
\noindent 3. adding a single red vertex $x_{2r}$ and single red edge $[x_{2r}, x_j]$ ($x_j$ remains purple).

Note that, for $j = 2r-1$, $\mathcal{G}_j$ contains as a subgraph the graph (we denote $\mathcal{G}_{1, \dots, 2r-2}$) consisting of the (purple) complete graph on the vertices $x_1, \dots, x_{2r-2}$, together with the black edges $[X_l, \overline{X_l}]$ for $1 \leq j \leq r-1$.  Hence, for each $1 \leq i,k \leq 2r-2$ where $i \neq k$, $\mathcal{G}_j$ also contains:
\begin{enumerate}
\item the locally smoothly embedded loop $L_{i,k}$ := $[x_i, \overline{x_k}, x_k, \overline{x_k}, x_k, \overline{x_i}, x_i, \overline{x_i}, x_i, x_k, \overline{x_k}, \overline{x_i}, x_i],$
starting with the purple edge $[x_i, \overline{x_k}]$ and ending with the black edge $[\overline{x_i}, x_i]$,
~\\
\vspace{-7.25mm}
\begin{figure}[H]
\centering
\includegraphics[width=1.8in]{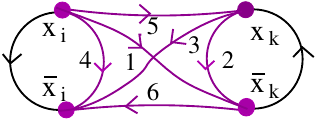}
\label{fig:ConstructionLoop}
\end{figure}
~\\
\vspace{-16mm}
\item and the smooth path $T_{i,k}$ := $[x_i, \overline{x_k}, x_k]$, traversing the purple edge $[x_i, \overline{x_k}]$, then black edge $[\overline{x_k}, x_k]$.
\end{enumerate}

\bigskip

\begin{lem}{\label{L:BirecurrentCompleteGraphltts}}
For each $1 \leq j \leq 2r-2$, $\mathcal{G}_j$ is an admissible $(r;(\frac{3}{2}-r))$ ltt structure.
\end{lem}

\begin{proof} The only step not following immediately from the definition and construction, is that $\mathcal{G}_j$ is birecurrent, thus admissible. To prove birecurrency we construct a locally smoothly embedded loop $\gamma$ traversing each edge of $\mathcal{G}_j$. Repeating $\gamma^j$ gives the locally smoothly embedded (periodic) line proving birecurrency. Since birecurrency is EPP-invariant, we assume $j=2r-2$.

Let $\gamma_1$ be the following smooth path traversing all edges of $\mathcal{G}_{1, \dots, 2r-2}$:
~\\
\vspace{-1mm}
$$L_{1,3} \ast L_{1,5} \ast \cdots \ast L_{1,2r-3} \ast T_{1,3} \ast L_{3,5} \ast \cdots \ast L_{3,2r-3} \ast T_{3,5} \ast L_{5,7} \ast \cdots \ast L_{5,2r-3} \ast \cdots \ast T_{2l-1,2l+1}\ast$$
$$L_{2l+1,2l+3} \ast \cdots \ast L_{2l+1,2r-3} \ast \cdots \ast  T_{2r-7,2r-5} \ast L_{2r-5,2r-3} \ast T_{2r-5,1}.$$
Remaining are the purple edges $[x_{2r-1}, x_i]$ for each $1 \leq i \leq 2r-2$, the red edge $[x_{2r}, x_j]$, and black edge $[x_{2r-1}, x_{2r}]$. $\gamma^j= \gamma_1\ast\gamma_2$, where $\gamma_2$ is a concatenation of the loops $[x_1, x_{2r}, x_{2r-1}, \overline{x_i}, x_i, x_2, x_1]$
~\\
\vspace{-7mm}
\begin{figure}[H]
\centering
\includegraphics[width=1.5in]{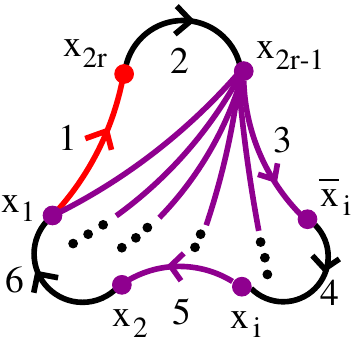}
\label{fig:ConstructionLoop} \\[-3mm]
\end{figure}

\noindent for $3\leq i\leq 2r-2$, together with the path $[x_1, x_{2r-1}, x_{2r}, x_1, x_{2r-3}, x_{2r-2}, x_2, x_1]$ to return to $x_1$:
$$\gamma_2=[x_1, x_{2r}, x_{2r-1}, \overline{x_3}, x_3, x_2, x_1] \ast \cdots \ast [x_1, x_{2r}, x_{2r-1}, \overline{x_i}, x_i, x_2, x_1] \ast \cdots$$
$$\cdots \ast [x_1, x_{2r}, x_{2r-1}, \overline{x_{2r-2}}, x_{2r-2}, x_2, x_1]\ast [x_1, x_{2r}, x_{2r-1}, x_2, x_1] \ast [x_1, x_{2r-1}, x_{2r}, x_1, x_{2r-3}, x_{2r-2}, x_2, x_1].$$

\end{proof}

\begin{thm}{\label{T:MainTheorem}} Let $\mathcal{C}_r$ denote the complete ($2r-1$)-vertex graph. For each $r \geq 3$, there exists an ageometric, fully irreducible $\phi \in Out(F_r)$ such that $\mathcal{C}_r$ is the ideal Whitehead graph $\mathcal{IW}(\phi)$ for $\phi$. \end{thm}

\begin{proof} We first give the representative $g$ with $\mathcal{IW}(g) \cong \mathcal{C}_3$, as this case is slightly different (ideal decomposition depicted below):
~\\
\vspace{-3mm}
$$
g =
\begin{cases} a \mapsto aba \bar{b} aac \bar{b} aba \bar{b} aacbaba \bar{b} aacaba \bar{b} aac \bar{b} a  \\
b \mapsto baba \bar{b} aac \bar{a} b \bar{c} \bar{a} \bar{a} b \bar{a} \bar{b} \bar{a} \bar{c} \bar{a} \bar{a} b \bar{a} \bar{b} \bar{a} \bar{b} \bar{c} \bar{a} \bar{a} b \bar{a} \bar{b} \bar{a} b  \\
c \mapsto aba \bar{b} aac
\end{cases}
$$
~\\
\vspace{-10mm}
\noindent \begin{figure}[H]
\centering
\includegraphics[width=6in]{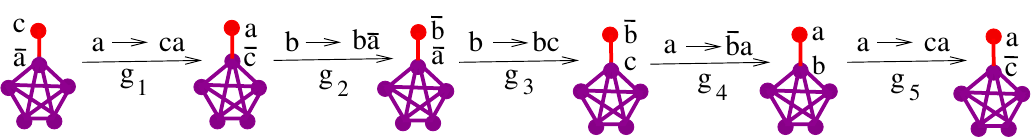}
\label{fig:CompleteGraph1}
\end{figure}
~\\
\vspace{-21mm}
\noindent \begin{figure}[H]
\centering
\includegraphics[width=6.5in]{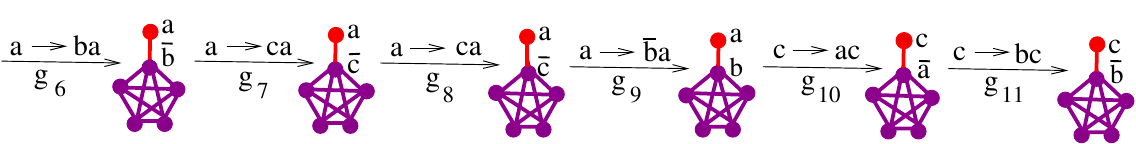}
\label{fig:CompleteGraph2}
\end{figure}
~\\
\vspace{-21mm}
\noindent \begin{figure}[H]
\centering
\includegraphics[width=5.6in]{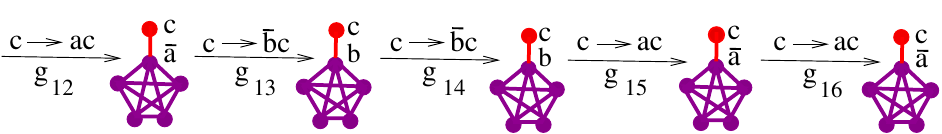}
\label{fig:CompleteGraph3} \\[-3mm]
\end{figure}

Suppose $r>3$. Let $\gamma_I$ be the construction loop in $\mathcal{G}_1$ starting with purple edge $[x_2, \overline{x_3}]=[x_2, x_4]$ and ending with purple edge $[x_{2r-2}, \overline{x_2}]=[x_{2r-2}, x_1]$ and black edge $[\overline{x_2}, x_2]=[x_1, x_2]$:
~\\
\vspace{-1mm}
$$L_{2,3} \ast L_{2,5} \ast \cdots \ast L_{2,2r-3} \ast T_{2,4} \ast L_{4,5} \ast L_{4,7} \ast \cdots \ast L_{4,2r-3} \ast T_{4,6} \ast L_{6,7} \ast L_{6,9} \ast \cdots \ast L_{6,2r-3} \ast T_{6,8} \ast \cdots$$
$$\cdots \ast T_{2l,2l+2}\ast L_{2l+2,2l+3}\ast T_{2l+2,2l+4} \ast \cdots \ast L_{2l+2,2r-3} \ast \cdots \ast T_{2r-6,2r-4} \ast L_{2r-4,2r-3} \ast T_{2r-4,2}.$$
Lemma \ref{L:BirecurrentCompleteGraphltts} implies, with a suitable switch, $\gamma_I$ induces an admissible construction composition $g_{\gamma_I}$. By Lemma \ref{P:PathConstruction}, the only purple edges left to construct are the $[x_{2r-1}, x_i]$ for each $1 \leq i \leq 2r-2$.

The initial ltt structure for the pure construction composition has red edge (red vertex first) $[x_{2r}, x_1]$. The switch for $g_{\gamma_I}$ will be determined by the purple edge $[x_2, x_3]$ and has initial ltt structure with red vertex $x_2$ and red edge $[x_2, x_3]$. In this ltt structure, the preimages of the purple edge left are $[x_{2r-1}, x_i]$, for each $1 \leq i \leq 2r-2$ with $i \neq 2$, and $[x_{2r-1}, x_{2r}]$. We ensure inclusion of all these edges except $[x_1, x_{2r-1}]$ by the construction loop $\gamma_{II}= T_{4,2r-1} \ast L_{2r-1,3} \ast L_{2r-1,5} \ast \cdots \ast L_{2r-1,2r-3} \ast T_{2r-1,4}.$ The red edge for the pure construction composition is (red vertex first) $[x_2, x_3]$.

The next switch is determined by the purple edge $[x_4, x_5]$ and has initial ltt structure with red edge $[x_4, x_5]$ and red vertex $x_4$. In this ltt structure, the preimage of the purple edge $[x_1, x_{2r-1}]$ left still looks like $[x_1, x_{2r-1}]$. We construct it using the construction loop $[x_6, x_2, x_1, x_{2r-1}, x_{2r}, x_5, x_6]$. The associated pure construction composition has initial ltt structure with red vertex $x_4$, and red edge $[x_4,x_5]$.

We use a sequence of switches to ensure that $g$ has PF transition matrix (see Lemma \ref{L:SwitchSequence}):
~\\
\vspace{-1mm}
$$G_6 \xrightarrow{g_7: x_{2r} \mapsto x_{2} x_{2r}} G_7 \xrightarrow{g_8: x_{2r-2} \mapsto x_{2r} x_{2r-2}} G_8 \xrightarrow{g_9: x_{2r-4} \mapsto x_{2r-2} x_{2r-4}} \dots$$
$$\xrightarrow{g_{r+3}: x_8 \mapsto x_{10} x_8} G_{r+3} \xrightarrow{g_{r+4}: x_6 \mapsto x_8 x_6} G_{r+4} \xrightarrow{g_{r+5}: x_4 \mapsto x_6 x_4} G_{r+5}$$
where, listed red vertex first:
$e^R_6=[x_{2},x_{2r-1}]$,
$e^R_7=[x_{2r},x_{1}]$,
$e^R_8=[x_{2r-2},x_{2r-1}]$,
$e^R_9=[x_{2r-4},x_{2r-3}]$,
$e^R_{r+2}=[x_{10},x_{11}]$,
$e^R_{r+3}=[x_8,x_9]$,
$e^R_{r+4}=[x_6,x_7]$, and
$e^R_{r+5}=[x_4,x_5]$.

We end with the Nielsen path prevention sequence (see Lemma \ref{L:NPKilling})
~\\
\vspace{-1mm}
$$G_0 \xrightarrow{g_1: x_5 \mapsto x_{2r} x_5} G_1 \xrightarrow{g_2: x_3 \mapsto x_5 x_3} G_2 \xrightarrow{g_3: x_2 \mapsto x_3 x_2} G_3 \xrightarrow{g_4: x_2 \mapsto x_{2r} x_2} G_4 \xrightarrow{g_5: x_2 \mapsto x_3 x_2} G_5 \xrightarrow{g_6: x_2 \mapsto x_{2r} x_2} G_6$$ where, listed red vertex first:
$e^R_0=[x_{2r},x_1]$,
$e^R_1=[x_5,x_{2r-1}]$,
$e^R_2=[x_3,x_6]$,
$e^R_3=[x_2,x_4]$,
$e^R_4=[x_2,x_{2r-1}]$,
$e^R_5=[x_2,x_4]$,
$e^R_6=[x_{2},x_{2r-1}]$.

Take an adequately high power so that all periodic directions are fixed.

By Lemma \ref{P:RepresentativeLoops}, it suffices that the constructed representatives are (1) pNp-free, (2) have PF transition matrices, and (3) have the appropriate ideal Whitehead graphs. Lemma \ref{L:NPKilling} implies (1), Lemma \ref{L:SwitchSequence} implies (2), and Lemma \ref{P:PathConstruction} implies (3). \qedhere
\end{proof}

\newpage

\bibliographystyle{amsalpha}
\bibliography{PaperReferences}

\end{document}